\documentclass[twoside,english,onefignum,onetabnum]{siamart190516}
\usepackage[T1]{fontenc}
\usepackage[latin9]{inputenc}
\usepackage{xcolor}
\usepackage{pdfcolmk}
\usepackage{mathrsfs}
\usepackage{enumitem}
\usepackage{amsmath}
\usepackage{graphicx}
\PassOptionsToPackage{normalem}{ulem}
\usepackage{ulem}

\makeatletter

\usepackage{amsfonts} 
\usepackage[version=3]{mhchem} 
\usepackage{tikz}
\usepackage{subfig}

\headers{Bounded reaction-diffusion systems}{J.M. Wentz and D.M. Bortz}

\title{Boundedness of a class of discretized reaction-diffusion systems\thanks{Submitted to arXiv on February 14th, 2020. 
		\funding{JMW is supported in part by an NSF GRFP and in part by the Interdisciplinary Quantitative Biology (IQ Biology) program at the BioFrontiers Institute, University of Colorado, Boulder. IQ Biology is generously supported by NSF IGERT grant number 1144807. DMB is supported by the NSF/NIH Joint DMS/NIGMS Mathematical Biology Initiative (R01GM126559)}}}

\author{Jacqueline M. Wentz\thanks{Department of Applied Mathematics, University of Colorado, Boulder, CO
		(\email{jacqueline.wentz@colorado.edu}, \email{dmbortz@colorado.edu}).}
	\and David M. Bortz\footnotemark[2]}

\definecolor{customgray}{RGB}{217,217,217}

\begin{document}
\maketitle
\begin{abstract}
Although the spatially continuous version of the reaction-diffusion
equation has been well studied, in some instances a spatially-discretized
representation provides a more realistic approximation of biological
processes. Indeed, mathematically the discretized and continuous systems
can lead to different predictions of biological dynamics. It is well
known in the continuous case that the incorporation of diffusion can
cause diffusion-driven blow-up with respect to the $L^{\infty}$ norm.
However, this does not imply diffusion-driven blow-up will occur in
the discretized version of the system. For example, in a continuous
reaction-diffusion system with Dirichlet boundary conditions and nonnegative
solutions, diffusion-driven blow up occurs even when the total species
concentration is non-increasing. For systems that instead have homogeneous
Neumann boundary conditions, it is currently unknown whether this
deviation between the continuous and discretized system can occur.
Therefore, it is worth examining the discretized system independently
of the continuous system. Since no criteria exist for the boundedness
of the discretized system, the focus of this paper is to determine
sufficient conditions to guarantee the system with diffusion remains
bounded for all time. We consider reaction-diffusion systems on a
1D domain with homogeneous Neumann boundary conditions and non-negative
initial data and solutions. We define a Lyapunov-like function and
show that its existence guarantees that the discretized reaction-diffusion
system is bounded. These results are considered in the context of
three example systems for which Lyapunov-like functions can and cannot
be found. 
\end{abstract}



\section{Introduction}

The reaction-diffusion (RD) modeling framework is used in biological
and ecological literature to understand how systems with spatial dependencies
evolve over time \cite{Perthame2015}. This equation is used, for
example, to describe spatial population dynamics and phenomena such
as pattern formation. Although the spatially-continuous RD equation
is often studied, in some instances the spatially-discretized system
allows for a more accurate representation of biological dynamics.
For example, the discretized system is used to model patchy habitats
in ecology \cite{Allen1987,Flather2002} and may prove useful in studying
the effects of metabolic compartmentalization \cite{Zecchin2015}.
It is worth investigating the discretized system because the dynamics
may differ significantly from the continuous system. For example in
RD systems with Dirichlet boundary conditions, even when the total
mass is conserved (i.e., the $L^{1}$ norm is bounded) the system
may blow up in $L^{\infty}$ \cite{Pierre2010}. This implies that
the discretized and continuous RD systems behave differently.

The existence of extensive literature discussing diffusion-driven
blow up for the spatially-continuous system demonstrates that the
question of boundedness is nontrivial (for a review see \cite{Fila2005}).
It is a known phenomenon that the addition of diffusion can affect
the stability of steady-states leading to, for example, pattern formation,
but the instability may also lead to unbounded solutions \cite{Fila2005,Murray2003,Wentz2018a,Marciniak-Czochra2016}.
Boundedness results for the continuous RD system have been obtained
using duality arguments, Sobolev embedding theorems, and Lyapunov-type
structures \cite{Morgan1990,Hollis1987,Melkemi2007,Fitzgibbon1997}.
However, to the best of our knowledge, our work is the first that
derives conditions to guarantee the discretized RD system is uniformly
bounded for all time. 

Our approach is based on previous work that used the existence of
a Lyapunov-type function to prove that the continuous RD system is
uniformly bounded \cite{Morgan1990}. The Lyapunov-type function was
required to be radially unbounded, additively separable, convex, and
decreasing along solution trajectories. Here, we define a similar
function, which we denote as a \emph{Lyapunov-like function (LLF),}
but replace the requirement of separability with more general conditions.
Ultimately, the criteria placed on the LLF allows us to obtain boundedness
results for systems with diffusion-driven instabilities. 

We examine systems that have two species reacting and diffusing, homogeneous
Neumann boundary conditions, and guaranteed non-negativity of solutions.
Homogeneous Neumann boundary conditions are arguably more realistic
for modeling compartments in biological systems then Dirichlet conditions.
Additionally, since we are interested in diffusion-driven blow up,
our focus will be on systems that have bounded solutions in reaction-only
case. Under these conditions, diffusion-driven blow up has been shown
to occur in the spatially-continuous system when the kinetics have
a globally stable steady-state \cite{Weinberger1999} and when there
is a clear biological application \cite{Lou2001} (i.e., modeling
mutualistic populations in ecology). 

In this work we present and prove sufficient conditions for the discretized
version of the RD system to be uniformly bounded over time. In Section~\ref{sec:notation-and-definitions}
we present relevant notation and define the properties of a LLF. In
Section~\ref{sec:boundedness-theorems} we prove that the existence
of this LLF guarantees the discretized RD system is uniformly bounded
over time. In Section~\ref{sec:illustative-examples} we consider
the results in the context of three examples that have well been well
studied in the spatially-continuous case. For the first example, the
continuous system has a bounded diffusion-driven instability, and,
in the second two examples, the continuous systems can blow up in
finite time. It is worth studying these examples in the discretized
setting because it is unknown whether, generally, the continuous and
discretized systems have the same boundedness properties. In Section
\ref{sec:discussion} we conclude with a discussion of other applications
and ideas for future directions.

\section{Notation and definitions\label{sec:notation-and-definitions}}

We are interested in RD systems on the normalized spatial interval
$I=[0,1]$ with two species $u$ and $v$. We will discretize this
system with respect to space by creating $n$ spatial compartments,
where $\mathscr{\mathcal{N}}=\{1,2,..,n\}$ represents the set of
compartment indices (Figure~\ref{fig:discretized-system}). 
\begin{figure}
\begin{tikzpicture}[scale=.8, every node/.style={scale=0.8}]
	\draw (0,0) -- (7,0);
	\draw (0,3) -- (7,3);
	\draw (8.5,0) -- (12.5,0);
	\draw (8.5,3) -- (12.5,3);
	\draw (0,0) -- (0,3);
	\draw (3,0) -- (3,3);
	\draw (6,0) -- (6,3);
	\draw (9.5,0) -- (9.5,3);
	\draw (12.5,0) -- (12.5,3);
	\node at (1.5,1.4){\underline{Reactions}};
	\node at (1.5,1){$f(u_1,v_1)$};
	\node at (1.5,0.6){$g(u_1,v_1)$};
	\node at (3,3.2){Diffusion};
	\node at (3, 2.6){$\ce{u_1 <-> u_2}$};
	\node at (3, 2.2){$\ce{v_1 <-> v_2}$};
	\node at (4.5,1.4){\underline{Reactions}};
	\node at (4.5,1){$f(u_2,v_2)$};
	\node at (4.5,0.6){$g(u_2,v_2)$};
	\node at (6,3.2){Diffusion};
	\node at (6, 2.6){$\ce{u_2 <-> u_3}$};
	\node at (6, 2.2){$\ce{v_2 <-> v_3}$};
	\node at (7.75,0){...};
	\node at (7.75,3){...};
	\node at (9.5,3.2){Diffusion};
	\node at (9.31, 2.6){$\ce{u_{n-1} <-> u_{n}}$};
	\node at (9.31, 2.2){$\ce{v_{n-1} <-> v_{n}}$};
	\node at (11,1.4){\underline{Reactions}};
	\node at (11,1){$f(u_n,v_n)$};
	\node at (11,0.6){$g(u_n,v_n)$};
	\node at (0,-.3) {$x_0$};
	\node at (3,-.3) {$x_1$};
	\node at (6,-.3) {$x_2$};
	\node at (9.5,-.3) {$x_{n-1}$};
	\node at (12.5,-.3) {$x_n$};
	\draw[<->] (0,-.6) -- (3,-.6);
	\node at (1.5,-.8) {$h$};
	\end{tikzpicture}
\centering{}\caption{Discretized system with $n$ spatial compartments. Reactions $f$
and $g$ occur within each compartment, and diffusion occurs between
adjacent compartments. \label{fig:discretized-system}}
\end{figure}
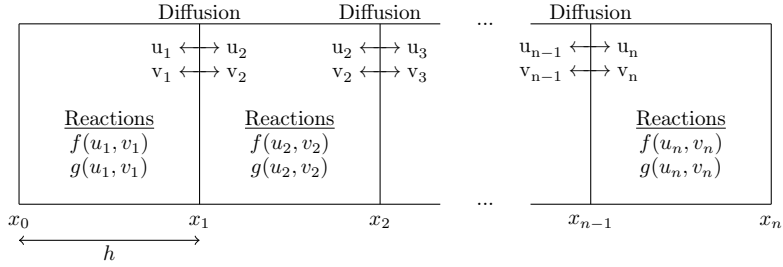
Let $h$ denote the uniform width of each compartment and $x_{0},x_{1},...,x_{n}$
denote the compartment edges (i.e., $x_{i}=ih$ where $h=1/n$). Let
$\mathbf{u}=(u_{1},u_{2},...,u_{n})^{T}$ and $\mathbf{v}=(v_{1},v_{2},...,v_{n})^{T}$
represent the average concentration of $u$ and $v$ in each of the
$n$ spatial compartments, where these concentrations are assumed
to be dimensionless. Let $\mathbf{f}(\mathbf{u},\mathbf{v})=(f(u_{1},v_{1}),f(u_{2},v_{2}),...,f(u_{n},v_{n}))^{T}$
and $\mathbf{g}(\mathbf{u},\mathbf{v})=(g(u_{1},v_{1}),g(u_{2},v_{2}),...,g(u_{n},v_{n}))^{T}$
represent the reactions taking place in each compartment. We will
model diffusion as a Fickian flux between two adjacent compartments
and, therefore, define 
\begin{equation}
D:=\left[\begin{array}{lllll}
\text{-}1 & 1 &  &  & 0\\
1 & \text{-}2 & 1\\
 & \ddots & \ddots & \ddots\\
 &  & 1 & \text{-}2 & 1\\
0 &  &  & 1 & \text{-}1
\end{array}\right]\label{eq:D}
\end{equation}
as the centered finite-difference matrix with homogeneous Neumann
boundary conditions. 

This leads to the following initial value problem 
\begin{equation}
\begin{aligned}\mathbf{u}_{t} & =\gamma\mathbf{f}(\mathbf{u},\mathbf{v})+\frac{1}{h^{2}}D\mathbf{u}\\
\mathbf{v}_{t} & =\gamma\mathbf{g}(\mathbf{u},\mathbf{v})+\frac{1}{h^{2}}dD\mathbf{v}\\
u_{i}(0) & =u_{i,0} &  & \text{for }i=1,2,...,n\\
v_{i}(0) & =v_{i,0} &  & \text{for }i=1,2,...,n
\end{aligned}
\label{eq:sysDisc}
\end{equation}
where $\gamma>0$ and $d>0$ are constants that are related to the
size of the domain and the diffusion coefficients (see \cite{Murray2003}
for a discussion of these parameters). To guarantee non-negativity
of solutions, we will require that $u_{i,0}\ge0$, $v_{i,0}\ge0$
for all $i\in\mathcal{N}$ and $f(0,v)\ge0$, $g(u,0)\ge0$ for all
$u,v\in[0,\infty)$. We will further require that $f$ and $g$ be
continuously differentiable. For simplicity we will assume that the
parameters are constant across space. This includes both reaction
parameters (i.e., constants within the functions $f$ and $g$) as
well as spatial parameters (i.e., $\gamma$ and $d$). Note that the
results can be easily generalized to systems with spatially varying
parameters. By the Picard-Lindelöf theorem, there is a $T_{max}>0$
such that a noncontinuable classical and unique solution to (\ref{eq:sysDisc})
exists for $t\in[0,T_{max})$ where it is possible that $T_{max}=\infty$.
Since the solution is classical, we know that $\mathbf{u}(t)$ and
$\mathbf{v}(t)$ are continuous for $t\in[0,T_{max})$ and if $T_{max}<\infty$,
then an element of $\mathbf{u}(t)$ and/or $\mathbf{v}(t)$ becomes
unbounded as $t\rightarrow T_{max}$. Thus, if the solution is bounded
for $t\in[0,T_{max})$, then $T_{max}=\infty$. 

Throughout the paper we will be using $\|\cdot\|:\mathbb{R}^{2}\rightarrow\mathbb{R}$
to represent the $l_{1}$-norm and we define the \emph{total species
concentration} as $\|(u,v)\|=u+v$. Furthermore, we will use variations
of $L$ (e.g., $L$, $\widetilde{L}$, $L_{i}$) to represent arbitrary
nonnegative constants.

\subsection{Lyapunov-like function\label{subsec:lyapunov-like-function} }

In this section we will define a \emph{Lyapunov-like function} (LLF)
for the reactions $f$ and $g$ given in (\ref{eq:sysDisc}). We will
later prove that the existence of this LLF guarantees that the discretized
RD system (\ref{eq:sysDisc}) is uniformly bounded for all time. The
classical definition of a Lyapunov function is a continuously differentiable,
locally positive-definite, scalar function that decreases along solution
trajectories in the neighborhood of a steady state. The LLF defined
here will instead decrease along reaction trajectories (i.e., solutions
when diffusion is not included) when the total species concentration
surpasses a threshold value. 

Throughout the paper we will use $W$ to denote a LLF. Let $W:\mathbb{R}_{\ge0}^{2}\rightarrow\mathbb{\mathbb{R}}_{\ge0}$
be a twice continuously differentiable function. To denote the partial
derivative of $W$ with respect to $u$ and $v$ we will use the notation
$\partial_{u}W$ and $\partial_{v}W$, respectively. We will use $(W)_{i}$,
$(\partial_{u}W)_{i}$, and $(\partial_{v}W)_{i}$ to denote the value
of $W$ and its partial derivatives evaluated in compartment $i$
(e.g., $(W)_{i}=W(u_{i},v_{i})$). Furthermore, we define the vectors
$\mathbf{W}:=((W)_{1},(W)_{2}...(W)_{n})$, $\partial_{u}\mathbf{W}:=((\partial_{u}W)_{1},(\partial_{u}W)_{2},...,(\partial_{u}W)_{n})$,
and $\partial_{v}\mathbf{W}:=((\partial_{v}W)_{1},(\partial_{v}W)_{2},...,(\partial_{v}W)_{n})$.

We say that $W$ is a LLF for the reactions $f$ and $g$ if $W$
satisfies five properties, denoted below as \ref{prop:decreases-with-time}--\ref{prop:finite-or-infinite-limits}.
These properties imply secondary properties on $W$, which we will
also present below. We first state three of the required properties,
i.e. \ref{prop:decreases-with-time}--\ref{prop:radially-unbounded}.
Notably \ref{prop:decreases-with-time} is the only property that
depends on the reactions $f$ and $g$.\bgroup 
\renewcommand\theenumi{(P\arabic{enumi})} 
\renewcommand\labelenumi{\theenumi}
\begin{enumerate}
\item \label{prop:decreases-with-time}There exists a $\text{\ensuremath{\underbar{K}}}>0$
such that if $\|(u,v)\|\ge\text{\ensuremath{\underbar{K}}}$ then
\begin{equation}
\left(\nabla W(u,v)\right)(f(u,v),g(u,v))^{T}\le0.\label{eq:decreases-with-time}
\end{equation}
\item \label{prop:convex} For all $(u,v)\in\mathbb{R}_{\ge0}^{2}$ the
second derivatives of $W$ are strictly positive, 
\begin{align*}
\partial_{uu}W(u,v) & >0,\quad\partial_{vv}W(u,v)>0,
\end{align*}
and the mixed partial derivative is non-negative
\[
\partial_{uv}W(u,v)\ge0.
\]
 
\item \label{prop:radially-unbounded} As the total species concentration
goes to infinity, the LLF approaches infinity:
\[
\lim_{\|(u,v)\|\rightarrow\infty}W(u,v)=\infty.
\]
\end{enumerate}
\egroup
\noindent The requirement that (\ref{eq:decreases-with-time}) holds only if
the total species concentration is large enough leads to a more complicated
boundedness proof but allows for LLFs to exist for systems that have
diffusion-driven instabilities. 

An additively-separable Lyapunov type function with properties similar
to \ref{prop:decreases-with-time}--\ref{prop:radially-unbounded}
was used to obtain a boundedness result in the continuous case \cite{Morgan1990}.
Notably requiring additive separability in addition to \ref{prop:decreases-with-time}--\ref{prop:radially-unbounded}
would be sufficient for proving the boundedness results in this paper.
However, we do not require the LLF to be additively separably because
it does not simplify the proofs significantly. Indeed, the final two
properties \ref{prop:level-set-tangent-lines}, \ref{prop:finite-or-infinite-limits}
are more general than requiring separability of the LLF. 

Before we present these final properties, we will provide some needed
notation and secondary properties that follow from \ref{prop:convex},
\ref{prop:radially-unbounded}. Variations of the letter $M$ (e.g.,
$M^{(L)}$, $M_{u}^{(L)}$, $M_{v}^{(L)}$) will be used to represent
a maximum value of either $W$ or a partial derivative of $W$ in
regions of $\mathbb{R}_{\ge0}^{2}$ where $\|(u,v)\|$ is constant.
For $L>0$ define
\begin{equation}
\begin{aligned}M^{(L)} & :=\max_{\|(u,v)\|=L}W(u,v)\\
M_{u}^{(L)} & :=\max_{\|(u,v)\|=L}\partial_{u}W(u,v)\\
M_{v}^{(L)} & :=\max_{\|(u,v)\|=L}\partial_{v}W(u,v).
\end{aligned}
\label{eq:M}
\end{equation}
For the partial derivatives of $W$, we will also consider what happens
in the limit as $u$ or $v$ approaches infinity. Thus, we define

\begin{equation}
\begin{aligned}M_{u}^{(\infty)}(v) & :=\lim_{u\rightarrow\infty}\partial_{u}W(u,v)\\
M_{v}^{(\infty)}(u) & :=\lim_{v\rightarrow\infty}\partial_{v}W(u,v).
\end{aligned}
\label{eq:Mu-Mv-infinity}
\end{equation}
By \ref{prop:convex} we know these limits either converge and exist
or diverge to infinity. 

In the following corollary we prove the existence of three additional
constants, $\underline{u},\underline{v}$, and $K$, that will be
used in Section \ref{sec:boundedness-theorems}.\bgroup 
\renewcommand\theenumi{(C\arabic{enumi})} 
\renewcommand\labelenumi{\theenumi}
\begin{corollary}
\label{cor:secondary-properties-1}Suppose $W:\mathbb{R}_{\ge0}^{2}\rightarrow\mathbb{\mathbb{R}}_{\ge0}$
is twice continuously differentiable. If $W$ satisfies \ref{prop:convex},
the following property holds:
\begin{enumerate}
\item \label{secondary-prop:Mu-L-Mv-L-increasing} The parameters $M_{u}^{(L)}$,
$M_{v}^{(L)}$ are monotonically increasing with respect to $L$. 
\end{enumerate}
If, in addition, $W$ satisfies \ref{prop:radially-unbounded}, then
the following properties also hold:
\begin{enumerate}[start=2]
\item \label{secondary-prop:u-v-underbar-exist}There exists constants
$\underline{u}$ and $\text{\ensuremath{\underline{v}}}$ such that
$\partial_{u}W(u,v)>0$ for all $u\ge\text{\ensuremath{\underline{u}}}$
and $\partial_{v}W(u,v)>0$ for all $v\ge\underline{v}$. 
\item \label{secondary-prop:K-exists} There exists a $K\ge\max\left\{ \text{\ensuremath{\underline{K}}},\text{\ensuremath{\underline{u}}},\underline{v}\right\} $
such that if $L<K$ then $M^{(L)}<M^{(K)}$.
\end{enumerate}
\end{corollary}
\noindent For the proof of this corollary see Appendix \ref{app-sec:proofOfC1andC2}.

For the fourth property, we will consider level-sets of $W$ and how
the tangent lines to the level-sets behave (Figure \ref{fig:level-set-tangent-lines-example}).
For every point $(u,v)\in\mathbb{R}_{\ge0}^{2}$ there exists a level
set of $W$ and corresponding tangent line that intersects $(u,v)$.
The following property describes how these tangent lines behave as
the total species concentration becomes large. 

\egroup
\bgroup 
\renewcommand\theenumi{(P\arabic{enumi})} 
\renewcommand\labelenumi{\theenumi}
\begin{enumerate}[resume, start=4]
\item \label{prop:level-set-tangent-lines} For a fixed value of $u$,
the level-set tangent lines do not become parallel to the $v$-axis
in the limit as $v\rightarrow\infty$, i.e.,
\[
\sup_{v\ge\text{\ensuremath{\underline{v}}}}\left|\frac{\partial_{u}W(u,v)}{\partial_{v}W(u,v)}\right|<\infty\text{ for all }u.
\]
 Similarly, for a fixed value of $v$, the level-set tangent lines
do not become parallel to the $u$-axis as $u\rightarrow\infty$,
i.e., 
\[
\begin{aligned}\sup_{u\ge\underline{u}} & \left|\frac{\partial_{v}W(u,v)}{\partial_{u}W(u,v)}\right|<\infty\text{ for all }v.\end{aligned}
\]
\end{enumerate}
From this property we immediately see that for all $L\ge0$ the following
constants exist and are finite:
\begin{equation}
\begin{aligned}R_{u,L} & :=\sup_{u\le L,v\ge\underline{v}}\left|\frac{\partial_{u}W(u,v)}{\partial_{v}W(u,v)}\right|\\
R_{v,L} & :=\sup_{v\le L,u\ge\underline{u}}\left|\frac{\partial_{v}W(u,v)}{\partial_{u}W(u,v)}\right|.
\end{aligned}
\label{eq:RuL-RvL}
\end{equation}
\begin{figure}
\centering{}\includegraphics[scale=0.85]{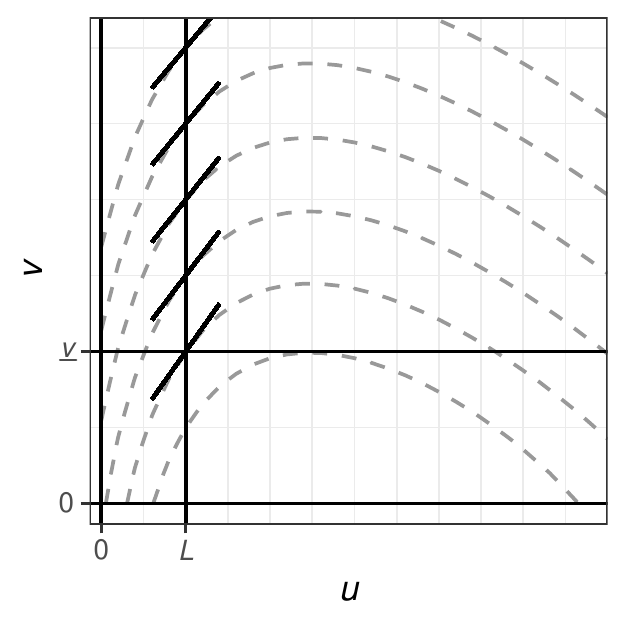}\includegraphics[scale=0.85]{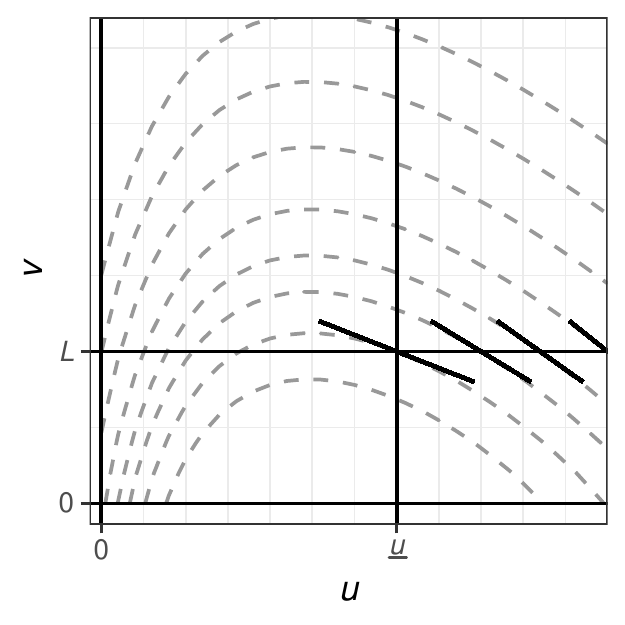}\caption{Graphical description of \ref{prop:level-set-tangent-lines}. The
dashed lines represent level sets of an example LLF and the black
line segments represent tangent lines. Left: At $u=L$ and $v\ge\text{\ensuremath{\underline{v}} }$the
level-set tangent lines do not approach vertical. Right: At $v=L$
and $u\ge\text{\ensuremath{\underline{u}} }$the level-set tangent
lines do not approach horizontal. \label{fig:level-set-tangent-lines-example}}
\end{figure}

For the final property, we place requirements on the limits of the
partial derivatives of the LLF.
\begin{enumerate}[resume]
\item \label{prop:finite-or-infinite-limits} For all $v\in[0,\infty)$,
either
\begin{enumerate}
\item $M_{u}^{(\infty)}(v)$ is finite and $\lim_{u\rightarrow\infty}\partial_{uv}W(u,v)$
exists and is finite, or 
\item $M_{u}^{(\infty)}(v)$ is infinite.
\end{enumerate}
Similarly, for all $u\in[0,\infty)$, either
\begin{enumerate}
\item $M_{v}^{(\infty)}(u)$ is finite and $\lim_{v\rightarrow\infty}\partial_{uv}W(u,v)$
exists and is finite, or 
\item $M_{v}^{(\infty)}(u)$ is infinite. 
\end{enumerate}
\end{enumerate}
\egroup{}
\bgroup 
\renewcommand\theenumi{(C\arabic{enumi})} 
\renewcommand\labelenumi{\theenumi} 

Now that we have stated all five properties, we will provide a formal
definition of a LLF.
\begin{definition}
\label{def:LLF}Let $W:\mathbb{R}_{\ge0}^{2}\rightarrow\mathbb{\mathbb{R}}_{\ge0}$
be a twice continuously differentiable function. Consider the discretized
system given by (\ref{eq:sysDisc}). Then $W$ is a LLF for this system
if \ref{prop:decreases-with-time}--\ref{prop:finite-or-infinite-limits}
are satisfied.
\end{definition}
Finally, we will show that $M_{u}^{(\infty)}(v)$ and $M_{v}^{(\infty)}(u)$,
given by (\ref{eq:Mu-Mv-infinity}), are constant functions and will
therefore be referred to as $M_{u}^{(\infty)}$ and $M_{v}^{(\infty)}$,
respectively.
\begin{corollary}
\label{cor:secondary-properties-2}The properties \ref{prop:convex},
\ref{prop:level-set-tangent-lines}, and \ref{prop:finite-or-infinite-limits}
imply the following secondary properties:
\begin{enumerate}[start=4]
\item \label{secondary-prop:Mu-infinity-constant} $M_{u}^{(\infty)}:=M_{u}^{(\infty)}(v)$
is independent of $v$ and $M_{u}^{(\infty)}>M_{u}^{(L)}$ for all
$L\in[0,\infty)$.
\item \label{secondary-prop:Mv-infinity-constant} $M_{v}^{(\infty)}:=M_{v}^{(\infty)}(u)$
is independent of $u$ and $M_{v}^{(\infty)}>M_{v}^{(L)}$ for all
$L\in[0,\infty)$.
\end{enumerate}
\end{corollary}
\egroup{}\noindent The proof of Corollary~\ref{cor:secondary-properties-2}
is given in Appendix \ref{app-sec:proofOfC1andC2}.

In the remainder of this paper, we will reference the constants $K$,
$\underline{u}$, and $\underline{v}$ given in Corollary~\ref{cor:secondary-properties-1}.
These constants only depend on the LLF. We will also reference the
constants $M_{u}^{(L)}$, $M_{v}^{(L)}$, $M^{(L)}$, $R_{u,L}$ and
$R_{v,L}$ given in (\ref{eq:M}), (\ref{eq:Mu-Mv-infinity}), and
(\ref{eq:RuL-RvL}). These constants depend on both the LLF and the
specified value of $L$. 

\subsection{Difference operator notation}

Let $\mathbf{w}=(w_{1},w_{2},...,w_{n})^{T}$ be an arbitrary vector
of length $n$. Define $\Delta_{i}^{+}$ and $\Delta_{i}^{-}$ as
the forward and backward difference operator, respectively, where
\begin{align*}
\Delta_{i}^{+}\mathbf{\mathbf{\mathbf{w}}} & :=w_{i+1}-w_{i}\text{ for }i=1,2,...,n-1\\
\Delta_{i}^{-}\mathbf{w} & :=w_{i}-w_{i-1}\text{ for }i=2,3,...,n.
\end{align*}
The $i$th element of the centered finite difference matrix, given
by (\ref{eq:D}), acting on a vector $\mathbf{w}$ is given as
\begin{equation}
(D\mathbf{w})_{i}=\begin{cases}
\Delta_{1}^{+}\mathbf{w}, & i=1\\
\Delta_{i}^{+}\mathbf{w}-\Delta_{i}^{-}\mathbf{w}, & i=2,3,...,n-1\\
-\Delta_{n}^{-}\mathbf{w}, & i=n.
\end{cases}\label{eq:Dw-i}
\end{equation}
Notice that for $i=1,n$ there is only a single term because we are
assuming homogeneous Neumann boundary conditions. We will apply these
difference operators to the species concentration vector, the LLF,
and the partial derivatives of the LLF. For example, $\Delta_{i}^{+}\mathbf{W}=(W)_{i+1}-(W)_{i}$.

\subsection{LLF $\Omega_{K}$ Region\label{subsec:LLF-omegaK-region}}

To prove the main result, we will consider the following region of
phase space (see Figure~\ref{fig:notation-MK-BK} for example):
\begin{figure}
\centering\includegraphics[scale=0.85]{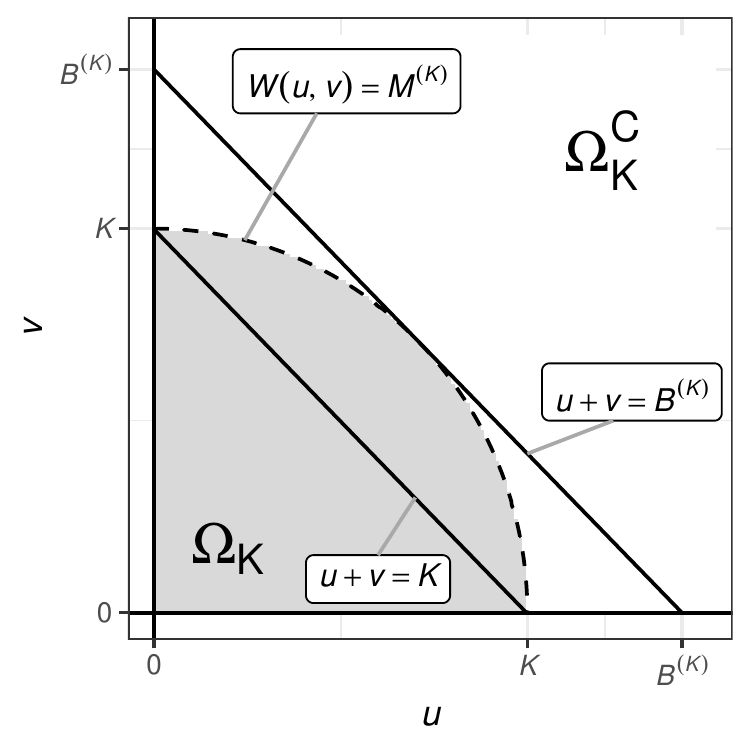} \caption{Illustration of example LLF $\Omega_{K}$ region and relevant constants.
The dashed line represents the level set $W(u,v)=M^{(K)}$ and the
shaded area is the $\Omega_{K}$ region.}
\label{fig:notation-MK-BK} 
\end{figure}
\[
\Omega_{K}:=\{(u,v)\in\mathbb{R}_{\ge0}^{2}\mid W(u,v)<M^{(K)}\}.
\]
Due to \ref{secondary-prop:K-exists} we know that $\Omega_{K}$ contains
all points within the bounded region defined by the level set $W(u,v)=M^{(K)}$.
Additionally, by \ref{prop:decreases-with-time} we know that outside
of this region, i.e. in $\Omega_{K}^{C}=\mathbb{R}_{\ge0}^{2}\setminus\Omega_{K}$,
(\ref{eq:decreases-with-time}) is satisfied. We will define the boundary
of $\Omega_{K}$ as 
\[
\partial\Omega_{K}:=\{(u,v)\in\mathbb{R}_{\ge0}^{2}\mid W(u,v)=M^{(K)}\}
\]
where by definition $\Omega_{K}\cap\partial\Omega_{K}=\emptyset$.
Define
\begin{equation}
B^{(K)}:=\max_{(u,v)\in\partial\Omega_{K}}\|(u,v)\|.\label{eq:BK}
\end{equation}
We then know that for any compartment $i$ such that $(u_{i},v_{i})\in\Omega_{K}$,
the total species concentration is bounded by $B^{(K)}$ (i.e., $\|(u_{i},v_{i})\|<B^{(K)}$).

\subsection{Notation for the sum of Lyapunov-like functions}

To consider the sum of LLFs across the spatial compartments, we will
partition the set of compartments into those that are and are not
contained in $\Omega_{K}$. Let the set $Y$ contain the indices of
compartments that are in $\Omega_{K}$ and $Y^{C}$ contain the indices
of compartments that are in $\Omega_{K}^{C}$ (i.e., $Y:=\{i\mid(u_{i},v_{i})\in\Omega_{K}\}$
and $Y^{C}=\mathcal{N}\setminus Y$ where recall that $\mathcal{N}=\{1,2,..,n\}$).
Figure \ref{fig:notation-Y-Z}a,c gives an example of this notation
(note that the sets $Z_{bdy}$ and $Z_{int}$ will be defined later
in this section). For notational simplicity we will say that a compartment
whose index is in $Y$ is a \emph{$Y$-compartment}, and similarly,
a compartment whose index is in $Y^{C}$ is a \emph{$Y^{C}$-compartment}. 

\begin{figure}[!t]
\begin{minipage}[b][1\totalheight][t]{0.45\columnwidth}%
\subfloat[Compartment location in phase space]{\includegraphics[scale=0.8]{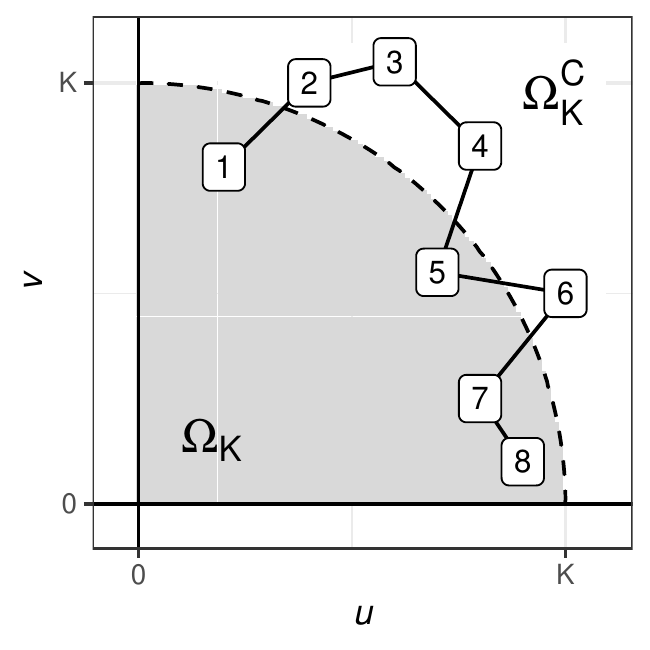}}%
\end{minipage}%
\begin{minipage}[b][1\totalheight][t]{0.45\columnwidth}%
\subfloat[Compartment location in physical space]{%
\noindent\begin{minipage}[t]{1\columnwidth}%
\hspace{.1cm}
\begin{tikzpicture}[scale=.68, every node/.style={scale=0.7}]
	\fill[customgray] (0,0) -- (1,0) -- (1,1) -- (0,1) -- (0,0);
	\fill[customgray] (4,0) -- (5,0) -- (5,1) -- (4,1) -- (4,0);
	\fill[customgray] (6,0) -- (7,0) -- (7,1) -- (6,1) -- (6,0);
	\fill[customgray] (7,0) -- (8,0) -- (8,1) -- (7,1) -- (7,0);
	\draw (0,0) -- (8,0) -- (8,1) -- (0,1) -- (0,0);
	\draw (1,0) -- (1,1);
	\draw (2,0) -- (2,1);
	\draw (3,0) -- (3,1);
	\draw (4,0) -- (4,1);
	\draw (5,0) -- (5,1);
	\draw (6,0) -- (6,1);
	\draw (7,0) -- (7,1);
	\node at (.5,.5){1};
	\node at (1,1.5){1};
	\draw (1,1.5) circle (.2);
	\draw [->] (1,1.3) -- (1,1.1);
	\node at (1.5,.5){2};
	\node at (2,1.5){2};
	\draw (2,1.5) circle (.2);
	\draw [->] (2,1.3) -- (2,1.1);
	\node at (2.5,.5){3};
	\node at (3,1.5){3};
	\draw (3,1.5) circle (.2);
	\draw [->] (3,1.3) -- (3,1.1);
	\node at (3.5,.5){4};	
	\node at (4,1.5){4};
	\draw (4,1.5) circle (.2);
	\draw [->] (4,1.3) -- (4,1.1);
	\node at (4.5,.5){5};	
	\node at (5,1.5){5};
	\draw (5,1.5) circle (.2);
	\draw [->] (5,1.3) -- (5,1.1);
	\node at (5.5,.5){6};	
	\node at (6,1.5){6};
	\draw [->] (6,1.3) -- (6,1.1);
	\node at (6.5,.5){7};	
	\node at (7,1.5){7};
	\draw (6,1.5) circle (.2);
	\draw [->] (7,1.3) -- (7,1.1);
	\node at (7.5,.5){8};
	\draw (7,1.5) circle (.2);
	\node at (3.5, -.5){$x \longrightarrow$};
\end{tikzpicture}%
\end{minipage}}

\subfloat[Set of compartment/edge indices]{%
\begin{minipage}[t]{0.95\columnwidth}%
\vspace{.3cm}
\small
\begin{align*}
Y & =\{1,5,7,8\}\\
Y^{C} & =\{2,3,4,6\}\\
Z_{bdy} & =\{1,4,5,6\}\\
Z_{int} & =\{2,3\}
\end{align*}
\vspace{-.4cm}%
\end{minipage}

}%
\end{minipage}

\caption{Example notation for an 8 compartment system. In (b) $Y^{C}$-compartments
are shaded in gray and the 7 edges are labeled by circles. There are
four edges that separate a $Y$ from a $Y^{C}$-compartment (i.e.,
$|Z_{bdy}|=4$) and two edges that separate two $Y^{C}$ compartments
(i.e., $|Z_{int}|=2$).}
\label{fig:notation-Y-Z} 
\end{figure}
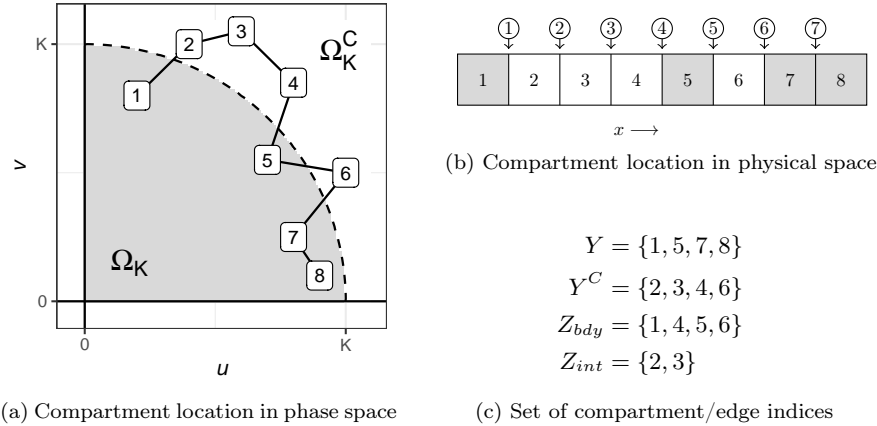

Suppose $X$ is either $Y$, $Y^{C}$, or $\mathcal{N}$ and define
the \emph{sum of LLFs over indices in $X$} as the function $W_{X}:\mathbb{R}^{n}\times\mathbb{R}^{n}\rightarrow\mathbb{R}$
where
\[
W_{X}(\mathbf{u},\mathbf{v}):=\sum_{i\in X}(W)_{i}.
\]
When $X:=\mathcal{N}$ this equation will be referred to as the \emph{System
LLF. }Notice that $W_{\mathcal{N}}=W_{Y}+W_{Y^{C}}\le nM^{(K)}+W_{Y^{C}}$\emph{.
}This implies that in order to bound $W_{\mathcal{N}}$, we need to
find a bound on $W_{Y^{C}}$\emph{.}

\subsection{The LLF Evolution Equation}

To bound $W_{Y^{C}}$, we will consider how its value evolves with
time. Note that this evolution is piecewise continuous where the discontinuities
occur when the membership of $Y^{C}$ changes. We will refer to such
an event as a \textit{crossing} of $\partial\Omega_{K}$. Compartment
$i$ will have undergone a crossing at time $t$ if $i\in Y^{C}(t)$
and $i\in Y(t\pm\epsilon)$ for an arbitrary small $\epsilon>0$.
If $i\in Y(t+\epsilon)$ then a crossing from $\Omega_{K}^{C}$ into
$\Omega_{K}$ has occurred and if $i\in Y(t-\epsilon)$, a crossing
from $\Omega_{K}$ into $\Omega_{K}^{C}$ has occurred. 

In the notation that follows, we consider time periods when there
are no crossings of $\partial\Omega_{K}$.We have that $W_{Y^{C}}$
is continuously differentiable and, therefore, we can define the \emph{LLF
Evolution Equation }as 
\[
\frac{dW_{Y^{C}}}{dt}=\left(\nabla_{\mathbf{u}}W_{Y^{C}}\right)\left(\frac{d\mathbf{u}}{dt}\right)^{T}+\left(\nabla_{\mathbf{v}}W_{Y^{C}}\right)\left(\frac{d\mathbf{v}}{dt}\right)^{T}
\]
where $\nabla_{\mathbf{u}}W_{Y^{C}}=(\partial_{u_{1}}W_{Y^{C}},\partial_{u_{2}}W_{Y^{C}},...,\partial_{u_{n}}W_{Y^{C}})$
and $\nabla_{\mathbf{v}}W_{Y^{C}}$ is defined analogously. The LLF
Evolution Equation tells us how the sum of LLFs over indices in $Y^{C}$
changes along solution trajectories.

We can divide the LLF evolution equation into reactive and diffusive
flux contributions, i.e.,

\[
\frac{dW_{Y^{C}}}{dt}=\gamma W_{Y^{C},R}+W_{Y^{C},D}
\]
where 
\begin{align}
W_{Y^{C},R} & :=\left(\nabla_{\mathbf{u}}W_{Y^{C}}\right)\mathbf{f}(\mathbf{u},\mathbf{v})^{T}+\left(\nabla_{\mathbf{v}}W_{Y^{C}}\right)\mathbf{g}(\mathbf{u},\mathbf{v})^{T}\label{eq:W-YCR}\\
W_{Y^{C},D} & :=\left(\nabla_{\mathbf{u}}W_{Y^{C}}\right)D\mathbf{u}^{T}+\left(\nabla_{\mathbf{v}}W_{Y^{C}}\right)dD\mathbf{v}^{T}\label{eq:W-YCD}
\end{align}
represent the reactive and diffusive contribution, respectively. 

\subsection{The diffusive fluxes and compartment edges}

We next examine the diffusive contribution to the LLF evolution equation,
(\ref{eq:W-YCD}), by considering a more refined set of fluxes. We
define an \emph{edge }as the boundary between two adjacent compartments
(see Figure \ref{fig:notation-Y-Z}b). Let $i=1,2,...,n-1$ denote
an edge where the $i$th edge represents the boundary between the
$i$ and $i+1$ compartment. We will call an edge that connects a
$Y$ with a $Y^{C}$-compartment a \emph{boundary edge }and an edge
that connects two $Y^{C}$-compartments an \emph{interior edge}. Note
that we do not name the edges that connect two $Y$-compartments. 

For each edge $i$ define the following

\begin{equation}
n_{i}=\mathbf{1}_{Y}(i+1)-\mathbf{1}_{Y}(i)\label{eq:ni}
\end{equation}
where $\mathbf{1}$ is the indicator function. We then have that the
sets

\begin{equation}
\begin{aligned}Z_{bdy:} & :=\left\{ i\in\mathcal{N}\setminus\{n\}\mid\left|n_{i}\right|=1\right\} \\
Z_{int} & :=\left\{ i\in\mathcal{N}\setminus\{n\}\mid n_{i}=0\text{ and }i\in Y^{C}\right\} 
\end{aligned}
\label{eq:Z-sets}
\end{equation}
contain the boundary edge indices and the interior edge indices, respectively.
The maximum sizes of $Z_{bdy}$ and $Z_{int}$ are both $n-1$. Figure
\ref{fig:notation-Y-Z}c defines $Z_{bdy}$ and $Z_{int}$ for an
example system.

With this notation in mind, let's again consider the diffusive contribution
to the LLF Evolution Equation and rewrite (\ref{eq:W-YCD}) as a sum
\begin{equation}
W_{Y^{C},D}=\sum_{i\in Y^{C}}\left(\partial_{u}W\right)_{i}\left(D\mathbf{u}\right)_{i}+d\left(\partial_{v}W\right)_{i}\left(D\mathbf{v}\right)_{i}.\label{eq:W-YCD-sum}
\end{equation}
Recall that $(D\mathbf{u})_{i}$ and $(D\mathbf{v})_{i}$ can be rewritten
as shown in (\ref{eq:Dw-i}) and rewrite (\ref{eq:W-YCD-sum}) as
\begin{align}
W_{Y^{C},D} & =\sum_{i\in Z_{bdy}}F_{bdy,i}+\sum_{i\in Z_{int}}F_{int,i}\label{eq:W-YCD-fluxes}
\end{align}
where
\begin{align}
F_{bdy,i} & :=n_{i}\left((\partial_{u}W)_{i+\frac{1-n_{i}}{2}}\Delta_{i}^{+}\mathbf{u}+d(\partial_{v}W)_{i+\frac{1-n_{i}}{2}}\Delta_{i}^{+}\mathbf{v}\right)\label{eq:Fbdyi}\\
F_{int,i} & :=-\left(\Delta_{i}^{+}\left(\partial_{u}\mathbf{W}\right)\Delta_{i}^{+}\mathbf{u}+d\Delta_{i}^{+}\left(\partial_{v}\mathbf{W}\right)\Delta_{i}^{+}\mathbf{v}\right).\label{eq:Finti}
\end{align}
We will refer to $F_{bdy,i}$ and $F_{int,i}$ as \emph{flux-effect
terms} because each represents the contribution of a single diffusive
flux to the LLF Evolution Equation. In Section~\ref{subsec:proofs-1}
we will show that, under certain conditions, the diffusive contribution
to the LLF Evolution Equation is negative.

\section{Boundedness theorems\label{sec:boundedness-theorems}}

We will prove the discretized RD system given by (\ref{eq:sysDisc})
is bounded if there exists a LLF for the reactions as described by
Definition \ref{def:LLF}. The proof involves two main steps. In Section~\ref{subsec:proofs-1},
we consider a snapshot of the system and show that if any compartment
exceeds a threshold total species concentration, then the solution
to the LLF Evolution Equation is decreasing. In Section~\ref{subsec:proofs-2},
we consider the evolution of the system, and show that the System
LLF is bounded. This bound on the System LLF in turn leads to a bound
on the concentration of species within a single compartment.

\subsection{At large species concentration the solution to the LLF Evolution
Equation is nonincreasing\label{subsec:proofs-1}}

Recall that the LLF Evolution Equation can be broken down into the
reactive (\ref{eq:W-YCR}) and diffusive (\ref{eq:W-YCD}) components.
We know that the reactive component is nonpositive due to \ref{prop:decreases-with-time}.
Thus, the main work of this section is to show that the diffusive
component is nonpositive when a compartment exceeds a threshold species
concentration. 

The outline of the results in this section is given as follows. Recall
that the diffusive component of the LLF Evolution Equation can be
rewritten as a summation of flux-effect terms, as given by (\ref{eq:W-YCD-fluxes}).
Each of these flux-effect terms can be bounded from above by a constant
(see Lemma~\ref{lemma:Fbdyi-bounded}). Therefore, the diffusive
component is negative if there exists one negative flux-effect term
with a sufficiently large magnitude. This negative flux-effect term
exists if two adjacent compartments have a large enough difference
between the amount of species they contain (see Lemma \ref{lemma:int-flux-effect-term-bbd-u}
and Corollary \ref{cor:int-flux-effect-term-bbd-v},~\ref{cor:int-flux-effect-term-bbd},
and~\ref{cor:bdy-flux-effect-term-bbd}). This difference in species
concentration is obtained if there is at least one $Y$-compartment
and one $Y^{C}$-compartment that exceeds a threshold species concentration
(Lemma~\ref{lemma:adj-comps-exist} and~\ref{lemma:diffusion-component-nonincreasing}).
It then immediately follows that the solution to the entire LLF Evolution
Equation is nonincreasing (Corollary~\ref{cor:solution-to-LLF-equation-nonincreasing}).

With this road-map in mind, we first show that the flux-effect terms
$F_{bdy,i}$ and $F_{int,i}$ have an upper bound. By \ref{prop:convex}
we immediately know that $F_{int,i}<0$ for all $i\in Z_{int}$. In
the next lemma, we prove that $F_{bdy,i}$ has an upper bound as well.
\begin{lemma}
\label{lemma:Fbdyi-bounded} Let $W$ by a LLF for the system given
by (\ref{eq:sysDisc}). Suppose at an arbitrary time $t>0$, $Z_{bdy}$
is nonempty. Pick $i\in Z_{bdy}$ and let $F_{bdy,i}$ be as given
in (\ref{eq:Fbdyi}). Then there exists a constant $F_{max}$ such
that $F_{bdy,i}\le F_{max}$. 
\end{lemma}
\begin{proof}
Pick $i\in Z_{bdy}$ and recall $n_{i}$ is given by (\ref{eq:ni}).
We rewrite $F_{bdy,i}$ as follows:
\[
F_{bdy,i}=n_{i}\left(F_{bdy,i}^{(u)}+F_{bdy,i}^{(v)}\right)
\]
where
\[
\begin{aligned}F_{bdy,i}^{(u)} & :=(\partial_{u}W)_{i+\frac{1-n_{i}}{2}}\Delta_{i}^{+}\mathbf{u},\hspace{1em}\hspace{1em}\hspace{1em} & F_{bdy,i}^{(v)} & :=d(\partial_{v}W)_{i+\frac{1-n_{i}}{2}}\Delta_{i}^{+}\mathbf{v}.\end{aligned}
\]
We will assume $n_{i}=1$ and note analogous logic can be applied
if $n_{i}=-1$. Since $n_{i}=1$, $i+1\in Y$ and therefore
\[
u_{i+1},v_{i+1},\Delta_{i}^{+}\mathbf{u},\Delta_{i}^{+}\mathbf{v}\le B^{(K)}
\]
where $B^{(K)}$ is given by (\ref{eq:BK}).

For notational simplicity we define the following two constants:
\begin{align*}
u^{*} & :=B^{(K)}\left(1+dR_{v,B^{(K)}}\right),\hspace{1em}\hspace{1em} & v^{*} & :=\frac{1}{d}B^{(K)}\left(d+R_{u,B^{(K)}}\right)
\end{align*}
where $R_{v,B^{(K)}}$ and $R_{u,B^{(K)}}$ are given by (\ref{eq:RuL-RvL}).
By \ref{secondary-prop:K-exists}, $B^{(K)}>\underline{u},\underline{v}$,
and thus $u^{*}\ge\underline{u}$ and $v^{*}\ge\underline{v}$.

We next consider three possible cases, one of which must occur. In
the first case we bound $F_{bdy,i}$ directly and in the second two
cases we bound $F_{bdy,i}^{(u)}$. First suppose $~\Delta_{i}^{+}\mathbf{u}\ge0$~and~$v_{i}>v^{*}$.
Using (\ref{eq:RuL-RvL}) and \ref{secondary-prop:u-v-underbar-exist},
we have that
\[
\begin{aligned}F_{bdy,i} & =(\partial_{v}W)_{i}\left(\frac{\left(\partial_{u}W\right)_{i}}{\left(\partial_{v}W\right)_{i}}\Delta_{i}^{+}\mathbf{u}+d\Delta_{i}^{+}\mathbf{v}\right)\\
 & \le d(\partial_{v}W)_{i}\left(B^{(K)}\left(R_{u,B^{(K)}}+d\right)-dv^{*}\right)\\
 & \le0
\end{aligned}
\]
Second, suppose $\Delta_{i}^{+}\mathbf{u}\ge0$~and~$v_{i}\le$$v^{*}$.
Using that $u_{i}<B^{(K)}$, we have that
\begin{align*}
F_{bdy,i}^{(u)} & =(\partial_{u}W)_{i}\Delta_{i}^{+}\mathbf{u}\\
 & \le\left(\max_{u\le B^{(K)},v\le v^{*}}\partial_{u}W(u,v)\right)B^{(K)}.
\end{align*}
Third suppose $\Delta_{i}^{+}\mathbf{u}<0$. If $(\partial_{u}W)_{i}\le0$
then $u_{i}<\underline{u}$ and by \ref{prop:convex} we have that
\begin{align*}
F_{bdy,i}^{(u)} & =\left|(\partial_{u}W)_{i}\right|\left|\Delta_{i}^{+}\mathbf{u}\right|\\
 & \le\text{\ensuremath{\left|\ensuremath{\partial_{u}W(0,0)}\right|}}\underline{u}.
\end{align*}
If instead $(\partial_{u}W)_{i}>0$, then $F_{bdy,i}^{(u)}<0$, and
the given bound still holds. 

These three cases imply that either
\[
F_{bdy,i}\le0\quad\quad\text{ or \ensuremath{\quad\quad}}F_{bdy,i}^{(u)}\le B^{(K)}\max_{u\le B^{(K)},v\le v^{*}}\left|\partial_{u}W(u,v)\right|.
\]
We can analogously show that either $F_{bdy,i}\le0$ or $F_{bdy,i}^{(v)}$
is bounded from above by a constant. This leads to the final result
that $F_{bdy,i}\le F_{max}$ where
\begin{equation}
F_{max}:=B^{(K)}\left(\max_{u\le B^{(K)},v\le v^{*}}\left|\partial_{u}W(u,v)\right|+d\max_{u\le u^{*},v\le B^{(K)}}\left|\partial_{v}W(u,v)\right|\right).\label{eq:Fmax}
\end{equation}
\end{proof}

Our next goal is to show that under certain conditions one of the
flux-effect terms that contributes to the LLF Evolution Equation,
i.e., $F_{bdy,i}$ or $F_{int,i}$, is smaller than an arbitrary negative
constant. We first pick an interior edge $\ell$ and show that the
desired result is obtained when $\left|\Delta_{\ell}^{+}\mathbf{u}\right|$
or $\left|\Delta_{\ell}^{+}\mathbf{v}\right|$ is sufficiently large
(Lemma \ref{lemma:int-flux-effect-term-bbd-u} and Corollary \ref{cor:int-flux-effect-term-bbd-v}).
Furthermore, if there is a large enough difference in the total species
concentration (i.e. $\left|\Delta_{\ell}^{+}(\textbf{\ensuremath{\mathbf{u}+\mathbf{v})}}\right|$)
the desired result is obtained (Corollary \ref{cor:int-flux-effect-term-bbd})
and similar results follow for an arbitrary boundary edge (Corollary
\ref{cor:bdy-flux-effect-term-bbd}). Below we will refer to compartment
$\ell+1$ as compartment $\ell^{+}$.
\begin{lemma}
\label{lemma:int-flux-effect-term-bbd-u} Let $W$ be a LLF for the
system given by (\ref{eq:sysDisc}). Pick an arbitrary time $t>0$
and suppose that $\ell\in Z_{int}$. Pick $A>0$ and $L>0$, where
\[
\min\left\{ \|(u_{\ell},v_{\ell})\|,\|(u_{\ell^{+}},v_{\ell^{+}})\|\right\} \le L.
\]
There exists a constant $G_{u}$ such that, if $\left|\Delta_{\ell}^{+}\mathbf{u}\right|\ge G_{u}$,
then the flux-effect term for interior edge $\ell$ is bounded from
above by $-A$, i.e. $F_{int,\ell}\le-A$ where $F_{int,\ell}$ is
given by (\ref{eq:Finti}) with $i:=\ell$.
\end{lemma}
\begin{proof}
Pick $\widetilde{L}>L$ such that $M_{u}^{(\widetilde{L})}>0$. By
\ref{secondary-prop:Mu-infinity-constant} this $\widetilde{L}$ exists
and there is a $\widetilde{u}$ such that $\partial_{u}W(\widetilde{u},0)=M_{u}^{(\widetilde{L})}.$
Define the following two constants
\begin{align*}
C_{1} & :=1-\frac{M_{u}^{(L)}}{M_{u}^{(\widetilde{L})}},\quad\quad C_{2}:=\left(R_{v,L}+\frac{\max_{\|(u,v)\|<L}|\partial_{v}W|}{M_{u}^{(\widetilde{L})}}\right)L
\end{align*}
and let
\begin{equation}
G_{u}:=\max\left\{ \widetilde{u}+L,\frac{A+dM_{u}^{(\widetilde{L})}C_{2}}{M_{u}^{(\widetilde{L})}C_{1}}\right\} .\label{eq:Gu}
\end{equation}
Notice that $C_{1},C_{2}>0$. The fact that $C_{1}>0$ follows from
\ref{secondary-prop:Mu-L-Mv-L-increasing}. 

Without loss of generality we will suppose that $\|(u_{\ell},v_{\ell})\|<\|(u_{\ell^{+}},v_{\ell^{+}})\|$
which implies $\|(u_{\ell},v_{\ell})\|<L$. Note that if $-\Delta_{\ell}^{+}\mathbf{u}\ge G_{u}$
then $u_{l}>G_{u}>L$, which is a contradiction. Therefore, $\Delta_{\ell}^{+}\mathbf{u}\ge G_{u}$
and by \ref{secondary-prop:Mu-L-Mv-L-increasing},

\begin{equation}
\max(0,M_{u}^{(L)})<M_{u}^{(\widetilde{L})}\le(\partial_{u}W)_{\ell^{+}}<M_{u}^{(\infty)}.\label{eq:Mu-inequalities}
\end{equation}

The flux effect term $F_{int,\ell}$ can be rewritten as follows:
\[
F_{int,\ell}=-(\partial_{u}W)_{\ell^{+}}\left(\frac{\Delta_{\ell}^{+}(\partial_{u}\mathbf{W})}{(\partial_{u}W)_{\ell^{+}}}\Delta_{\ell}^{+}\mathbf{u}+d\frac{\Delta_{\ell}^{+}(\partial_{v}\mathbf{W})}{(\partial_{u}W)_{\ell^{+}}}\Delta_{\ell}^{+}\mathbf{v}\right).
\]
We will next examine the two terms in the parentheses that contribute
to $F_{int,\ell}$. For the first term, using (\ref{eq:Mu-inequalities})
gives the following bound:
\begin{align*}
\frac{\Delta_{l}^{+}\left(\partial_{u}\mathbf{W}\right)}{(\partial_{u}W)_{\ell^{+}}}\Delta_{\ell}^{+}\mathbf{u} & =\left(1-\frac{(\partial_{u}W)_{\ell}}{(\partial_{u}W)_{\ell^{+}}}\right)\Delta_{\ell}^{+}\mathbf{u}\ge\left(1-\frac{M_{u}^{(L)}}{M_{u}^{(\widetilde{L})}}\right)G_{u}=C_{1}G_{u}
\end{align*}
For the second term, first suppose that $\Delta_{\ell}^{+}\mathbf{v}\le0$.
This implies that $v_{\ell^{+}},\left|\Delta_{\ell}^{+}\mathbf{v}\right|\le L$
and leads to the following bound:
\begin{equation}
\begin{aligned}d\frac{\Delta_{\ell}^{+}\left(\partial_{v}\mathbf{W}\right)}{(\partial_{u}W)_{\ell^{+}}}\Delta_{\ell}^{+}\mathbf{v} & =-d\left(\frac{\left(\partial_{v}W\right)_{\ell^{+}}}{(\partial_{u}W)_{\ell^{+}}}-\frac{\left(\partial_{v}W\right)_{\ell}}{(\partial_{u}W)_{\ell^{+}}}\right)\left|\Delta_{\ell}^{+}\mathbf{v}\right|\\
 & \ge-d\left(R_{v,L}+\frac{\max_{\|(u,v)\|<L}|\partial_{v}W|}{M_{u}^{(\widetilde{L})}}\right)L=-dC_{2}.
\end{aligned}
\label{eq:largeueffect2}
\end{equation}
If instead $\Delta_{\ell}^{+}\mathbf{v}>0$, then the left hand side
of (\ref{eq:largeueffect2}) is positive and, therefore, the bound
still holds. The positivity of the left hand side follows from (\ref{eq:Mu-inequalities}),
which implies $(\partial_{u}W)_{\ell^{+}}>0$ and \ref{prop:convex},
which implies $\Delta_{\ell}^{+}\left(\partial_{v}\mathbf{W}\right)\ge0$.
Finally, using (\ref{eq:Mu-inequalities})--(\ref{eq:largeueffect2})
we have that 
\[
F_{int,\ell}\le-M_{u}^{(\widetilde{L})}\left(C_{1}G_{u}-dC_{2}\right)\le-A.
\]
\end{proof}

\begin{corollary}
\label{cor:int-flux-effect-term-bbd-v} Suppose the assumptions of
Lemma \ref{lemma:int-flux-effect-term-bbd-u} hold. Given $A>0$ and
$L>0$ where 
\[
\min\left\{ \|(u_{\ell},v_{\ell})\|,\|(u_{\ell^{+}},v_{\ell^{+}})\|\right\} \le L,
\]
there exists a $G_{v}\ge L$ such that, if $\left|\Delta_{\ell}^{+}\mathbf{v}\right|\ge G_{v}$,
then the flux-effect term for interior edge $\ell$ is bounded from
above, i.e., $F_{int,\ell}\le-A$ where $F_{int,\ell}$ is given by
(\ref{eq:Finti}) with $i:=\ell$.
\end{corollary}
\begin{proof}
The proof follows using the same logic as the proof to Lemma~\ref{lemma:int-flux-effect-term-bbd-u}.
First, find $\widetilde{L}>L$ such that $M_{v}^{(\widetilde{L})}>0$
and $\widetilde{v}$ such that $\partial_{v}(0,\widetilde{v})=M_{v}^{(\widetilde{L})}.$
Then define an analogous set of constants

\[
C_{3}:=1-\frac{M_{v}^{(L)}}{M_{v}^{(\widetilde{L})}},\quad\quad C_{4}:=L\left(R_{u,L}+\frac{\max_{\|(u,v)\|<L}|\partial_{u}W(u,v)|}{M_{v}^{(\widetilde{L})}}\right)
\]
and let

\begin{equation}
G_{v}:=\max\left\{ \widetilde{v}+L,\frac{A+M_{v}^{(\widetilde{L})}C_{4}}{dM_{v}^{(\widetilde{L})}C_{3}}\right\} .\label{eq:Gv}
\end{equation}
We then have that $F_{int,\ell}\le-M_{v}^{(\widetilde{L})}\left(-C_{4}+dC_{3}G_{v}\right)\le-A.$
\end{proof}

\begin{corollary}
\label{cor:int-flux-effect-term-bbd} Suppose the assumptions of Lemma~\ref{lemma:int-flux-effect-term-bbd-u}
hold. Given $A>0$ and $L>0$ where 
\[
\min\left(\|(u_{\ell},v_{\ell})\|,||(u_{\ell^{+}},v_{\ell^{+}})\|\right)\le L,
\]
there exists a $G\ge L$ such that if
\begin{equation}
\Bigl|||(u_{\ell^{+}},v_{\ell^{+}})\|-\|(u_{\ell},v_{\ell})\|\Bigr|\ge G,\label{eq:G-inequality}
\end{equation}
then $\max\left\{ \|(u_{\ell},v_{\ell})\|,\|(u_{\ell^{+}},v_{\ell^{+}})\|\right\} >B^{(K)}$
and the flux-effect term for interior edge $\ell$ is bounded from
above as follows
\begin{equation}
F_{int,\ell}\le-A\label{eq:required-inequality}
\end{equation}
where $F_{int,\ell}$ is given by (\ref{eq:Finti}) with $i:=\ell$.
\end{corollary}
\begin{proof}
Let

\begin{equation}
G=\max\left\{ 2G_{u},2G_{v},B_{K}\right\} .\label{eq:G}
\end{equation}
where $G_{u}$ is given by (\ref{eq:Gu}), $G_{v}$ is given by (\ref{eq:Gv}),
and $B_{K}$ is given by (\ref{eq:BK}). Notice that
\[
\Bigl|||(u_{\ell^{+}},v_{\ell^{+}})\|-\|(u_{\ell},v_{\ell})\|\Bigr|=\left|\Delta_{\ell}^{+}\mathbf{u}+\Delta_{\ell}^{+}\mathbf{v}\right|
\]
and, therefore, (\ref{eq:G-inequality}) and (\ref{eq:G}) imply that
either $\left|\Delta_{\ell}^{+}\mathbf{u}\right|\ge G_{u}$ or $\left|\Delta_{\ell}^{+}\mathbf{v}\right|\ge G_{v}$.
Thus, we apply either Lemma~\ref{lemma:int-flux-effect-term-bbd-u}
or Corollary~\ref{cor:int-flux-effect-term-bbd-v} to show that (\ref{eq:required-inequality})
holds. Finally, since $G\ge B_{K}$, using (\ref{eq:G-inequality}),
we have that $\max\left\{ \|(u_{\ell},v_{\ell})\|,\|(u_{\ell^{+}},v_{\ell^{+}})\|\right\} >B^{(K)}$.
\end{proof}

\begin{corollary}
\label{cor:bdy-flux-effect-term-bbd} Suppose the assumptions of Lemma~\ref{lemma:int-flux-effect-term-bbd-u}
hold where instead we pick $\ell\in Z_{bdy}$. Given $A>0$ and $L>0$
where $\min\left(\|(u_{\ell},v_{\ell})\|,||(u_{\ell^{+}},v_{\ell^{+}})\|\right)\le L$,
find the $G$ from Corollary~\ref{cor:int-flux-effect-term-bbd},
given by (\ref{eq:G}). If (\ref{eq:G-inequality}) holds, then
\[
\max\left\{ \|(u_{\ell},v_{\ell})\|,\|(u_{\ell^{+}},v_{\ell^{+}})\|\right\} >B^{(K)}
\]
 and the flux-effect term for boundary edge $\ell$ is bounded from
above as follows

\begin{equation}
\begin{aligned}F_{bdy,\ell} & \le-A\end{aligned}
\label{eq:required-inequality-2}
\end{equation}
where $F_{bdy,\ell}$ is given by (\ref{eq:Fbdyi}) with $i:=\ell$.
\end{corollary}
\begin{proof}
The result follows directly from Corollary~\ref{cor:int-flux-effect-term-bbd}.
Without loss of generality again suppose $\|(u_{\ell^{+}},v_{\ell^{+}})\|>\|(u_{\ell},v_{\ell})\|.$
It immediately follows that $\|(u_{\ell^{+}},v_{\ell^{+}})\|\ge G\ge B^{(K)}$.
This, in turn, implies that $\ell^{+}\in Y^{C}$ and hence $\ell\in Y$.
The equation for $F_{bdy,\ell}$ then reduces to
\[
F_{bdy,\ell}=-(\partial_{u}W)_{\ell^{+}}\Delta_{\ell}^{+}\mathbf{u}-d(\partial_{v}W)_{\ell^{+}}\Delta_{\ell}^{+}
\]
To bound this equation, we apply the logic from Lemma~\ref{lemma:int-flux-effect-term-bbd-u},
Corollary~\ref{cor:int-flux-effect-term-bbd-v} and \ref{cor:int-flux-effect-term-bbd}
where $(\partial_{u}W)_{\ell}$ and $(\partial_{v}W)_{\ell}$ are
equal to zero. The result of this logic gives us that $F_{bdy,\ell}\le-A$.
\end{proof}

We will next assume the set $Y$ is not empty and show that when a
threshold total species concentration is passed in at least one compartment,
we can find a interior or boundary edge that satisfies either (\ref{eq:required-inequality})
or (\ref{eq:required-inequality-2}), respectively.
\begin{lemma}
\label{lemma:adj-comps-exist} Let $W$ be a LLF for the system given
by (\ref{eq:sysDisc}) and pick an arbitrary time $t>0$. Pick $A>0$
and suppose there exists a compartment $k$ such that $k\in Y$. Then
there is a threshold concentration $C>0$ such that if $\max(\|(u_{i},v_{i})\|)_{i=1}^{n}\ge C$,
then there exists an interior or boundary edge, $\ell$, that satisfies
either (\ref{eq:required-inequality}) or (\ref{eq:required-inequality-2}),
respectively. 
\end{lemma}
\begin{proof}
Suppose there does not exist an interior edge or boundary edge that
satisfies either (\ref{eq:required-inequality}) or (\ref{eq:required-inequality-2}),
respectively. The species concentration in compartment $k$ is bounded
such that $\|(u_{k},v_{k})\|<B^{(K)}$. We will apply either Corollary~\ref{cor:int-flux-effect-term-bbd}
or \ref{cor:bdy-flux-effect-term-bbd} to iteratively bound the compartment
concentration for $i=k+1,k+2,...,n$ and $i=k-1,k-2,...,1$. 

Define $L_{k}:=B^{(K)}$ and, for $i=k+1,...,n$, iteratively find
the $G_{i}$ given by (\ref{eq:G}) in Corollary~\ref{cor:int-flux-effect-term-bbd}
where $L:=L_{i-1}$. Next, set $L_{i}:=G_{i}+L_{i-1}$. If $\|(u_{i},v_{i})\|-\|(u_{i-1},v_{i-1})\|>G_{i}$
then edge $i$ is either an interior or boundary edge and, by Corollary~\ref{cor:int-flux-effect-term-bbd}
or \ref{cor:bdy-flux-effect-term-bbd} either (\ref{eq:required-inequality})
or (\ref{eq:required-inequality-2}) holds, resulting in a contradiction.
Therefore, $\|(u_{i},v_{i})\|-\|(u_{i-1},v_{i-1})\|\le G_{i}=L_{i}-L_{i-1}$,
implying $\|(u_{i},v_{i})\|\le L_{i}$. Thus, we have obtained an
upper bound for the compartment $i$ and can continue the iteration.
Similarly, for $i=k-1,...,1$, the same methodology can be used to
generate compartment bounds, where we find the $G_{i}$ given by (\ref{eq:G})
where $L:=L_{i+1}$ . 

This logic leads to the following bound
\[
\|(u_{k+i},v_{k+i})\|<C
\]
for $i=-k,...,-2,-1,1,2,..,n-k$ where the threshold species concentration
$C$ is 
\begin{equation}
C:=B^{(K)}+\sum_{i=1,i\ne k}^{n}G_{i}.\label{eq:C}
\end{equation}
Suppose that there exists a compartment $m$ such that $\|(u_{m},v_{m})\|\ge C$.
We have a contradiction, and therefore there exists at least one interior
or boundary edge such that either (\ref{eq:required-inequality})
or (\ref{eq:required-inequality-2}) holds. 
\end{proof}

In the next lemma we prove that if a compartment exceeds a threshold
species concentration then the diffusive contribution to the LLF Evolution
Equation is nonpositive (i.e., $W_{Y^{C},D}\le0$).
\begin{lemma}
\label{lemma:diffusion-component-nonincreasing} Let $W$ be a LLF
for the system given by (\ref{eq:sysDisc}). Define $A:=nF_{max}$
where $F_{max}$ is given by (\ref{eq:Fmax}) in Lemma~\ref{lemma:Fbdyi-bounded}.
If we use this $A$ as the constant in Lemma $\ref{lemma:adj-comps-exist}$,
we can let $C$ be as defined in (\ref{eq:C}). If at an arbitrary
time $t>0$, $\max_{i}\|(u_{i},v_{i})\|\ge C$, then $W_{Y^{C},D}\le0.$
\end{lemma}
\begin{proof}
First suppose that $Y$ is empty. Then $Y^{C}=\mathcal{N}$ and
\[
W_{Y^{C}}=\sum_{i=1}^{n}(W)_{i}.
\]
The diffusion component of the LLF Evolution Equation is given as
\[
\begin{aligned}W_{Y^{C},D} & =\left(\nabla_{u}W_{Y^{C}}\right)D\mathbf{u}^{T}+d\left(\nabla_{v}W_{Y^{C}}\right)D\mathbf{v}^{T}\\
 & =\sum_{i=1}^{n-1}-\Delta_{i}^{+}(\partial_{u}\mathbf{W})\Delta_{i}^{+}\mathbf{u}-d\Delta_{i}^{+}(\partial_{v}\mathbf{W})\Delta_{i}^{+}\mathbf{v}\\
 & \le0.
\end{aligned}
\]
This result follows from property \ref{prop:convex} of the LLF.

Next, suppose that $Y$ is not empty. By Lemma~\ref{lemma:adj-comps-exist},
there exists either an interior edge that satisfies (\ref{eq:required-inequality})
or a boundary edge that satisfies (\ref{eq:required-inequality-2}).
We will let $\ell$ denote the index of this edge. We know $\max\left\{ \|(u_{\ell},v_{\ell})\|,\|(u_{\ell^{+}},v_{\ell^{+}})\|\right\} >B^{(K)}$.
Without loss of generality, suppose $\|(u_{\ell^{+}},v_{\ell^{+}})\|>\|(u_{\ell},v_{\ell})\|$
and thus $\ell^{+}\in Y^{C}$. We will next consider two possible
cases: $\ell\in Z_{int}$ or $\ell\in Z_{bdy}$.

First let's assume $\ell\in Z_{int}$. By Corollary~\ref{cor:int-flux-effect-term-bbd},
we have that 
\[
F_{int,\ell}\le-nF_{max}.
\]
Using that $F_{int,i}<0$ and $F_{bdy,i}\le F_{max}$ for all $i$,
and that $Z_{bdy}$ has at most $n-1$ elements, we have that 
\[
\begin{aligned}W_{Y^{C},D} & =\sum_{i\in Z_{bdy}}F_{bdy,i}+\sum_{i\in Z_{int}}F_{int,i}\\
 & \le\sum_{i\in Z_{bdy}}F_{bdy,i}+F_{int,\ell}\\
 & \le(n-1)F_{max}-nF_{max}\\
 & <0.
\end{aligned}
\]
Next, assume $\ell\in Z_{bdy}$. We have that by Corollary~\ref{cor:bdy-flux-effect-term-bbd}
\[
F_{bdy,\ell}<-nF_{max}.
\]
We then can similarly bound $W_{Z,D}$ as follows,
\[
\begin{aligned}W_{Y^{C},D} & =\sum_{i\in Z_{bdy}}F_{bdy,i}+\sum_{i\in Z_{int}}F_{int,i}\\
 & \le F_{bdy,\ell}+\sum_{i\in Z_{bdy},i\ne\ell}F_{bdy,i}\\
 & \le-nF_{max}+(n-2)F_{max}\\
 & <0.
\end{aligned}
\]
\end{proof}

Finally, we'll prove the main result of this section. Specifically,
we will next show that given a compartment exceeds a threshold species
concentration, the solution to the LLF Evolution Equation is decreasing.
\begin{corollary}
\label{cor:solution-to-LLF-equation-nonincreasing} If the assumptions
of Lemma~\ref{lemma:diffusion-component-nonincreasing} hold, then
$\frac{dW_{Y^{C}}}{dt}\le0$. 
\end{corollary}
\begin{proof}
From Lemma~\ref{lemma:diffusion-component-nonincreasing} we know
that $W_{Y^{C},D}\le0$. Furthermore, (\ref{eq:W-YCR}) can be rearranged
as follows 
\[
W_{Y^{C},R}=\sum_{i\in Y^{C}}\left(\nabla W(u_{i},v_{i})\right)(f(u_{i},v_{i}),g(u_{i},v_{i}))^{T}.
\]
This sum is clearly less than zero by \ref{prop:decreases-with-time}
of the LLF. Therefore, 
\[
\frac{dW_{Y^{C}}}{dt}=\gamma W_{Y^{C},R}+W_{Y^{C},D}\le0.
\]
\end{proof}

\subsection{The System LLF is bounded\label{subsec:proofs-2}}

In this section we will consider how the \emph{System LLF }or the
sum of LLFs across all compartments evolves with time. We will first
only consider times when one compartment exceeds a threshold species
concentration. During these times we will consider what occurs to
the System LLF when the membership of $Y$ and $Y^{C}$ changes. Below
we define a time-dependent function that bounds the System LLF. We
will show that this function decreases with time if any compartment
exceeds a threshold total species concentration (Lemma~\ref{lemma:W-decreasing}).
We conclude by considering all times during which a solution exists,
i.e., $t\in[0,T_{max})$, and show that the given function will always
bound the system. Hence, the total species concentration in each compartment
is uniformly bounded over time (Theorem~\ref{thm:mainresult}). Furthermore,
since the solution does not blow up, we are guaranteed that $T_{max}=\infty$.

Consider the discretized RD system given by (\ref{eq:sysDisc}) and
let $W:\mathbb{R}_{\ge0}^{2}\rightarrow\mathbb{R}$ be a LLF for the
reactions. Using this LLF and the initial data, we define the threshold
species concentration $\underline{C}$ as follows. 
\begin{definition}
\label{def:C-thres}Find the $F_{max}$ for the LLF given by (\ref{eq:Fmax})
in~Lemma~\ref{lemma:Fbdyi-bounded}. Apply Lemma~\ref{lemma:adj-comps-exist}
where $A:=nF_{max}$ to find $C$. We then define the threshold species
concentration as 
\[
\text{\ensuremath{\underline{C}}}=\max\left\{ \left(\|(u_{i,0},v_{i,0})\|\right)_{i=1}^{n},C\right\} .
\]
\end{definition}
We will consider times during which the total species concentration
in at least one compartment is greater than or equal to $\underline{C}$
and examine what occurs to the sum of LLFs when crossings of $\partial\Omega_{K}$
occur. Let $n_{t}$ be the number of compartments contained in $\Omega_{K}$
at time $t$ and define 
\begin{equation}
\mathcal{W}(t)=n_{t}M^{(K)}+W_{Y^{C}(t)}(\mathbf{u}(t),\mathbf{v}(t))\label{eq:mathcalW}
\end{equation}
Note that $\mathcal{W}(t)$ provides an upper bound on the System
LLF at time $t$. 

We will first show that while the species concentration in at least
one compartment exceeds $\underline{C}$, $\mathcal{W}(t)$ is a decreasing
function of time. To do this we consider a closed interval of time,
and allow for crossings of $\partial\Omega_{K}$ to occur at either
end of the interval.
\begin{lemma}
\label{lemma:W-decreasing} Pick $\tau_{1}$ and $\tau_{2}$ such
that no crossings occur for $t\in\left(\tau_{1},\tau_{2}\right)$.
Suppose for $t\in[\tau_{1},\tau_{2}]$, $\max\left\{ \left(\|(u_{i},v_{i})\|\right)_{i=1}^{n}\right\} \ge\underline{C}$.
Then, 
\begin{equation}
\mathcal{W}\left(\tau_{2}\right)\le\mathcal{W}\left(\tau_{1}\right).\label{eq:Wt1-less-than-Wt2}
\end{equation}
\end{lemma}
\begin{proof}
We will first show how $W_{Y^{C}(\tau_{1})}(\mathbf{u}(\tau_{1}),\mathbf{v}(\tau_{1}))$
and $W_{Y^{C}(\tau_{2})}(\mathbf{u}(\tau_{2}),\mathbf{v}(\tau_{2}))$
are related. Notice that $Y(t)$ and $Y^{C}(t)$ do not have changes
in membership for all $t\in\left(\tau_{1},\tau_{2}\right)$ since
no crossings of $\partial\Omega_{K}$ occur. Pick $t\in(\tau_{1},\tau_{2})$.
Corollary~\ref{cor:solution-to-LLF-equation-nonincreasing} along
with the continuity of $\mathbf{u}$ and $\mathbf{v}$ imply that
\begin{equation}
W_{Y^{C}(t)}(\mathbf{u}(\tau_{2}),\mathbf{v}(\tau_{2}))\le W_{Y^{C}(t)}(\mathbf{u}(\tau_{1}),\mathbf{v}(\tau_{1})).\label{eq:W-inequality-at-edges}
\end{equation}
Let $n_{in}$ denote the number of compartments that cross into $\Omega_{K}$
at time $\tau_{1}$, and let $n_{out}$ denote the number of compartments
that cross out of $\Omega_{K}$ at time $\tau_{2}.$ We have that
\[
W_{Y^{C}\left(\tau_{1}\right)}\left(\mathbf{u}\left(\tau_{1}\right),\mathbf{v}\left(\tau_{1}\right)\right)=W_{Y^{C}(t)}\left(\mathbf{u}\left(\tau_{1}\right),\mathbf{v}\left(\tau_{1}\right)\right)+n_{in}M^{(K)}
\]
and 
\[
W_{Y^{C}\left(\tau_{2}\right)}\left(\mathbf{u}\left(\tau_{2}\right),\mathbf{v}\left(\tau_{2}\right)\right)=W_{Y^{C}(t)}\left(\mathbf{u}\left(\tau_{2}\right),\mathbf{v}\left(\tau_{2}\right)\right)+n_{out}M^{(K)}
\]
Using these relations and (\ref{eq:W-inequality-at-edges}), we deduce
the following inequality: 
\[
W_{Y^{C}\left(\tau_{2}\right)}\left(\mathbf{u}\left(\tau_{2}\right),\mathbf{v}\left(\tau_{2}\right)\right)\le W_{Y^{C}\left(\tau_{1}\right)}\left(\mathbf{u}\left(\tau_{1}\right),\mathbf{v}\left(\tau_{1}\right)\right)+(n_{out}-n_{in})M^{(K)}.
\]

Using (\ref{eq:mathcalW}) and $n_{\tau_{2}}=n_{\tau_{1}}+n_{in}-n_{out}$
gives us the final result since 
\[
\begin{aligned}\mathcal{W}\left(\tau_{2}\right) & =\left(n_{\tau_{1}}+n_{in}-n_{out}\right)M^{(K)}+W_{Y^{C}(\tau_{2})}\left(\mathbf{u}\left(\tau_{2}\right),\mathbf{v}\left(\tau_{2}\right)\right)\\
 & \le n_{\tau_{1}}M^{(K)}+W_{Y^{C}(\tau_{1})}\left(\mathbf{u}\left(\tau_{1}\right),\mathbf{v}\left(\tau_{1}\right)\right)\\
 & \le\mathcal{W}\left(\tau_{1}\right).
\end{aligned}
\]
\end{proof}

Finally, we will prove our main result by considering how the system
evolves for all time $t\in[0,T_{max})$.
\begin{theorem}
\label{thm:mainresult} Consider the ODE system given by (\ref{eq:sysDisc}),
and suppose there exists a LLF $W:\mathbb{R}_{\ge0}^{2}\rightarrow\mathbb{R}$
for this system Then, there exists an upper bound $B>0$ such that
$\|(u_{i}(t),v_{i}(t))\|\le B$ for $i=1,2,...,n$ and all $t\in[0,\infty)$. 
\end{theorem}
\begin{proof}
Pick $t\in[0,T_{max})$. Define $T^{(M)}$ as follows:
\[
T^{(M)}=\left\{ t\in[0,T_{max})\mid\max_{i=1,..,n}\left\{ \|(u_{i}(t),v_{i}(t))\|\right\} \ge\underline{C}\right\} 
\]
where $\underline{C}$ is given by Definition~\ref{def:C-thres}.
By the continuity of $u_{i}(t)$ and $v_{i}(t)$ we know that $T^{(M)}$
is a finite union of closed connected sets where 
\[
T^{(M)}=\bigcup_{i=1}^{m}T_{i}^{(M)}.
\]

Pick $i=1,2,..,m$ and suppose there are $J$ crossings in $T_{i}^{(M)}$.
Again, by the continuity of $u_{i}(t)$ and $v_{i}(t)$ we know that
the number of crossings is finite. Let $t_{j}$, for $j=1,2,...,J$,
denote the time at which crossings of $\partial\Omega_{K}$ occur
and let $t_{0}$ and $t_{J+1}$ denote the start and end of the time
interval, respectively (i.e., $T_{i}^{(M)}=[t_{0},t_{J+1}]$). We
then have that 
\[
T_{i}^{(M)}=\bigcup_{j=0}^{J}[t_{j},t_{j+1}].
\]

Pick $j=0,1,...,J$ and pick $\tau\in[j,j+1]$. Apply Lemma~\ref{lemma:W-decreasing}
to show 
\[
\begin{aligned}W_{\mathcal{\mathcal{N}}}(\mathbf{u}(\tau),\mathbf{v}(\tau)) & \le\mathcal{W}(\tau)\le\mathcal{W}(t_{j})\le\mathcal{W}(t_{j-1})\le...\\
 & \le\mathcal{W}(t_{0})=n_{t_{0}}M^{(K)}+W_{Y^{C}(t_{0})}(\mathbf{u}(t_{0}),\mathbf{v}(t_{0}))
\end{aligned}
\]
From the definition of $t_{0}$, we know for all $i$, $\|(u_{i}(t_{0}),v_{i}(t_{0}))\|\le\underline{C}$,
and, therefore, $W(u_{i}(t_{0}),v_{i}(t_{0}))\le M^{(\underline{C})}$.
This follows from \ref{prop:convex}, \ref{secondary-prop:K-exists},
and the fact that $\underline{C}>K$ . The final bound we obtain is
\begin{equation}
W_{\mathcal{N}}(\mathbf{u}(\tau),\mathbf{v}(\tau))\le n_{t_{0}}M^{(K)}+(n-n_{t_{0}})M^{(\underline{C})}\le nM^{(\underline{C})}.\label{eq:system-LLF-bounded}
\end{equation}
Note that since $j$ and $i$ were arbitrary the same bound holds
for all $\tau\in T^{(M)}$. Furthermore, for $\tau\le t$ such that
$\tau\notin T^{(M)}$ we know that $W(u_{i}(t_{0}),v_{i}(t_{0}))\le M^{(\underline{C})}$
for $i=1,2,...,n.$ Thus, the bound given by (\ref{eq:system-LLF-bounded})
still holds. Finally, since $t$ was arbitrary this bound holds for
all $t\in[0,T_{max})$.

Define $\Omega=\{(u,v)\mid W(u,v)\le nM^{(\underline{C})}\}$. By
(\ref{eq:system-LLF-bounded}) $(u_{i}(t),v_{i}(t))\in\Omega$ for
all $i\in\{1,2,..,n\}$ and $t\in[0,T_{max})$. This implies that
the total species concentration in each compartment is bounded by
$B$ where 
\[
B=\max_{(u,v)\in\Omega}\|(u,v)\|
\]
and $T_{max}=\infty$.
\end{proof}

\section{Applications: General results and example systems\label{sec:illustative-examples}}

Here, we discuss how to apply the boundedness results from Section
\ref{sec:boundedness-theorems}. We first present some general rules
that can be used to help determine whether an LLF exists for a specific
reaction set. We then present three example systems that illustrate
how diffusion-driven instability can lead to both bounded and unbounded
solutions. These examples illustrate that, although diffusion-driven
blow up can occur in the spatially-discretized system, Theorem \ref{thm:mainresult}
can be applied to find systems for which this is not the case.

\subsection{General rules for determining whether a LLF exists\label{subsec:LLF-existence}}

In some cases it is possible to quickly find a candidate LLF or to
show an LLF cannot exist. If a known Lyapunov function exists for
the reactions, it may also satisfy the requirements of a LLF. By definition,
any global Lyapunov function satisfies \ref{prop:decreases-with-time}
and \ref{prop:radially-unbounded}. Therefore, it remains to show
that a specific Lyapunov function satisfies \ref{prop:convex}, \ref{prop:level-set-tangent-lines},
and \ref{prop:finite-or-infinite-limits}. In the following corollary,
we show that if the Lyapunov function is additively separable and
satisfies \ref{prop:convex} then \ref{prop:level-set-tangent-lines}
and \ref{prop:finite-or-infinite-limits} follow.
\begin{corollary}
\label{cor:LF-is-LLF}Let $W:\mathbb{R}_{\ge0}^{2}\rightarrow\mathbb{R}_{\ge0}$
by a Lyapunov function for the reactions $f(u,v)$ and $g(u,v)$.
If, in addition $W$ is additively separable (i.e., $W(u,v)=w_{1}(u)+w_{2}(v)$)
and satisfies \ref{prop:convex}, then $W$ is a LLF.
\end{corollary}
The proof of this corollary is given in \ref{sec:General-rules-proofs}.
As an example application of Corollary~\ref{cor:LF-is-LLF}, consider
a Lyapunov function consisting of a positive monomial for $u$ and
$v$ with degree $\ge2$ (i.e., $W(u,v)=\alpha u^{m}+\beta v^{p}$
where $\alpha,\beta\in\mathbb{R_{+}}$ and $m,p=2,3,...$). It follows
immediately from Corollary \ref{cor:LF-is-LLF} that $W$ is a LLF.

In some cases, it is possible to quickly determine that no LLF exists
for a system. In the next corollary, we present conditions on the
reactions $f$ and $g$ in (\ref{eq:sysDisc}) that guarantee a LLF
for the system cannot be found.
\begin{corollary}
\label{cor:no-LLF-exists}Consider the system given by (\ref{eq:sysDisc}).
If for any $v\ge0$ the following conditions are satisfied,
\begin{equation}
\begin{aligned}\liminf_{u\rightarrow\infty}f(u,v) & >0,\hspace{1em}\hspace{1em} & \lim_{u\rightarrow\infty}\left|\frac{f(u,v)}{g(u,v)}\right| & =\infty,\end{aligned}
\label{eq:f-positive}
\end{equation}
then there does not exist an LLF for the system. Analogously, if for
any $u\ge0$ the following conditions are satisfied,
\begin{equation}
\begin{aligned}\liminf_{v\rightarrow\infty}g(u,v) & >0,\hspace{1em}\hspace{1em} & \lim_{v\rightarrow\infty}\left|\frac{g(u,v)}{f(u,v)}\right| & =\infty,\end{aligned}
\label{eq:g-positive}
\end{equation}
then there does not exist an LLF for the system.
\end{corollary}
The proof of this corollary is given in \ref{sec:General-rules-proofs}.

\subsection{Example bounded and unbounded discretized RD systems\label{subsec:example-systems}}

We consider three example systems. The first is a set of chemical
reactions that is biologically realistic \cite{Murray2002}, the second
is a set of strongly mutualistic populations in ecology \cite{Lou2001},
and the third, although more abstract, demonstrates that diffusion-driven
blow up can occur when the reaction-only system has a globally stable
steady-state \cite{Weinberger1999}. 

Consider the following set of chemical reactions involving the species
$U$ and $V$, studied in \cite{Schnakenberg1979,Murray2002,Murray2003}:

\begin{align}
	\ce{A ->[k_1] U}, \quad
	\ce{B ->[k_2] V}, \quad
	\ce{2U + V ->[k_3] 3U}, \quad
	\ce{U ->[k_4] \emptyset}
\end{align}Here, $A$ and $B$ are positive source terms and the $k_{i}$'s represent
kinetic constants. If we make the system nondimensional the following
reactions are obtained
\begin{equation}
\begin{aligned}f(u,v) & =a-u+u^{2}v\\
g(u,v) & =b-u^{2}v
\end{aligned}
\label{eq:reactions-bounded}
\end{equation}
where $a$ and $b$ are positive constants (see \cite{Murray2002}).
Since $f$ and $g$ are continuously differentiable and $f(0,v),g(u,0)\ge0$,
these reactions satisfy the requirements stated following (\ref{eq:sysDisc}).
Without diffusion, there is either a periodic solution or a stable
steady state, implying the reaction-only system is bounded. The continuous
RD system, however, exhibits diffusion-driven instability \cite{Murray2002}. 

Let's consider the following LLF candidate 
\begin{equation}
W(u,v)=u+2v+\frac{c}{u+1}+\frac{1}{v+1}\label{eq:example-LLF}
\end{equation}
where $c>0$ is a constant that depends on the parameters $a$ and
$b$. This function has the desired properties, \ref{prop:decreases-with-time}--\ref{prop:finite-or-infinite-limits}
(see Appendix \ref{app-sec:exampleSystemBounded} for proofs). Therefore, by
Theorem~\ref{thm:mainresult}, the discretized system given by (\ref{eq:sysDisc})
is bounded for all time. These results are confirmed numerically for
a two-compartment system (Figure~\ref{fig:example-phase-planes},
left panel).

\begin{figure}
\begin{centering}
\includegraphics[clip,scale=0.7]{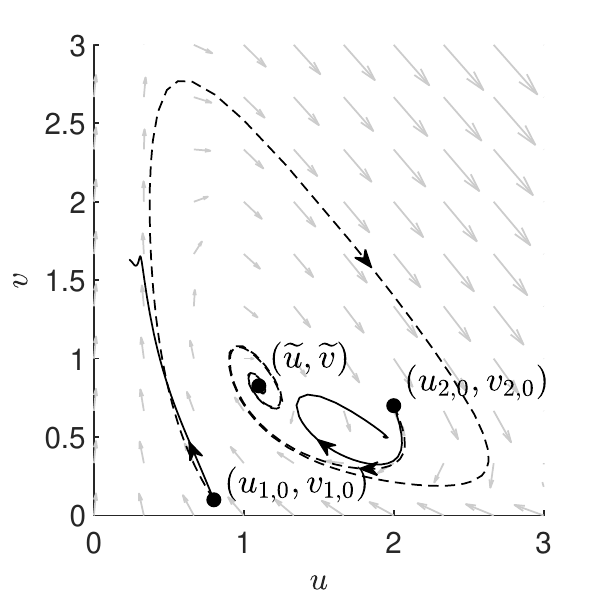}\includegraphics[scale=0.7]{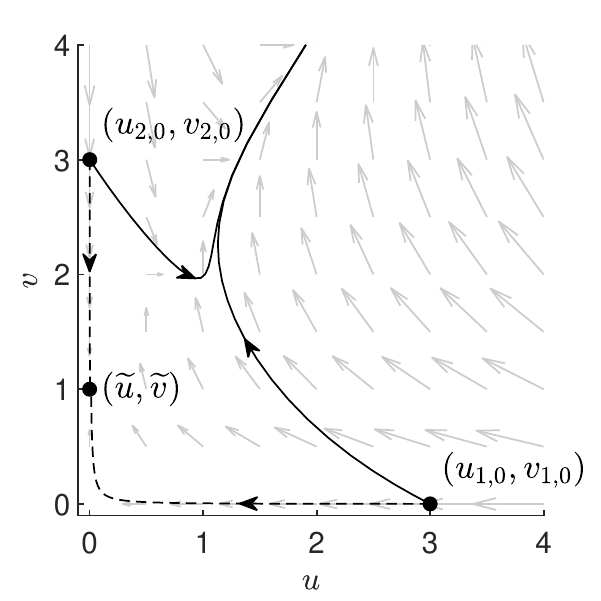}\includegraphics[scale=0.7]{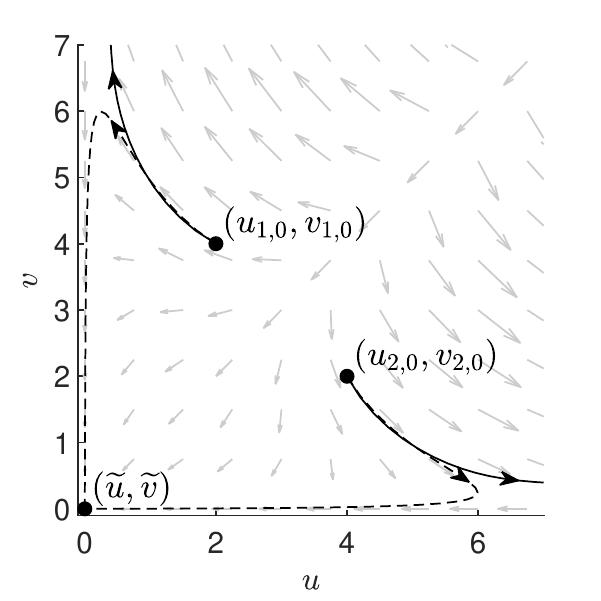}
\par\end{centering}
\caption{Example trajectories for two compartment systems where diffusion alters
the steady state and leads to bounded (left panel) and unbounded (middle
and right panel) solutions. The solid and dashed lines show the evolution
of the species concentration in each compartment with and without
diffusion, respectively. The gray arrows represent the vector field
$(f,g)$ given by (\ref{eq:reactions-bounded}), (\ref{eq:reactions-mutual}),
and (\ref{eq:reactions-unbounded}) for the left, middle, and right
panels, respectively. For complete details on simulations see Appendix \ref{sec:simulation-info}.
\label{fig:example-phase-planes}}
\end{figure}

The second example of a strongly mutualistic population has the following
reactions \cite{Lou2001}
\begin{equation}
\begin{aligned}f(u,v) & \begin{aligned}=u(a_{1}-b_{1}u+c_{1}v)\end{aligned}
\\
g(u,v) & =v(a_{2}+b_{2}u-c_{2}v)
\end{aligned}
\label{eq:reactions-mutual}
\end{equation}
where $a_{1},a_{2}\in\mathbb{R}$ and $b_{1},b_{2},c_{1},c_{2}\in\mathbb{R_{+}}$.
Here $a_{1}$ and $a_{2}$ represent growth or death rates of each
species, $c_{1}$ and $b_{2}$ represent the positive mutual interaction
between the two species, and the $b_{1}$ and $c_{2}$ represent growth
limitations a species exerts on itself \cite{Pao1993}. For strong
mutualism we require $b_{2}c_{1}>b_{1}c_{2}.$ For some initial conditions
the reaction-only system approaches a stable steady state while the
spatially continuous RD system has unbounded solutions (see Proposition
1.1 and Theorem 1.2 in \cite{Lou2001}). A LLF does not exist for
this system (see Appendix \ref{sm-subsec:noLLFforExample} for proof) which
suggests the discretized RD system becomes unbounded. Numerical simulations
confirm this result (Figure \ref{fig:example-phase-planes}, middle
panel).

The third and final example we will discuss was first introduced by
\cite{Weinberger1999}. The reactions are given as
\begin{equation}
\begin{aligned}f(u,v) & =uv(u-v)(u+1)-\delta u\\
g(u,v) & =uv(v-u)(v+1)-\delta v
\end{aligned}
\label{eq:reactions-unbounded}
\end{equation}
where $\delta$ is a positive constant. Again, these reactions satisfy
the stated requirements in Section~\ref{sec:notation-and-definitions}
since $f$ and $g$ are continuously differentiable and $f(0,v),g(u,0)\ge0$.
This system has a single steady state at the origin and a corresponding
global Lyapunov function, 
\[
V(u,v)=(u+1)^{2}(v+1)^{2}.
\]
This existence of this Lyapunov function proves that the steady state
in the reaction-only system is globally stable. 

For this system the conditions of Corollary~\ref{cor:no-LLF-exists}
are satisfied and therefore an LLF does not exist (see Appendix \ref{sm-subsec:noLLFforExample}).
Again however, this does not imply that the system is unbounded. Simulations
of the two compartment system suggest that the system becomes unbounded
as $t\rightarrow\infty$ (Figure \ref{fig:example-phase-planes},
right panel), and in Appendix \ref{app-subsec:exampleUnboundedProof} we prove
this result for small $\delta$.

\section{Discussion}

\label{sec:discussion}

We have defined a class of spatially discretized RD systems with a
uniform boundedness property. We looked specifically at systems with
two species reacting and diffusing on a 1D domain with homogeneous
Neumann boundary conditions and guaranteed positivity of solutions.
This RD system must additionally have a Lyapunov-like function (LLF)
as described by Definition~\ref{def:LLF}. Under these conditions,
we are guaranteed that the total concentration of species in the system
is bounded for all time. Notably, the existence of a LLF for a system
only depends on the reactions, and is therefore independent of the
domain size and diffusion rates of the two species in the system. 

The results presented here are generalizable to systems with reaction
parameters (i.e., parameters within the functions $f$ and $g$) and
diffusion parameters (i.e., $d$ and $\gamma$) that vary across space.
This generalization allows us to consider a broader range of systems.
For example, parameter values could follow spatial gradients or the
diffusion rate between two compartments could be altered to represent
a physical barrier, such as a membrane. In a system with spatially
varying reaction parameters, each individual compartment would have
its own parameter set. For example, in (\ref{eq:reactions-bounded})
the reactions occurring in compartment $i$ would have parameters
$a_{i}$ and $b_{i}$. If an LLF exists that satisfies \ref{prop:decreases-with-time}
in each spatial compartment, then our results can be generalized to
prove that the system is bounded. To allow diffusion to vary across
space, we would define a diffusion value of $u$ and $v$ across each
edge in the system. The result of this change would cause the flux
effect-terms given by (\ref{eq:Fbdyi}) and (\ref{eq:Finti}) to depend
on these edge-dependent diffusion values. In principle, the same logic
in the proofs would hold, where the bounds obtained would now depend
on the maximum and minimum values of the diffusion parameters. Rigorously
proving these result is a topic of future research.

Using a LLF to prove boundedness provides a method for examining any
mathematical description of a biological system. For example, biochemical
dynamics can be described mathematically using mass-action kinetics
\cite{Voit2015}, Hill Functions \cite{Goutelle2008}, and Michaelis-Menten
Kinetics \cite{Johnson2011}. We specifically showed how the results
can be applied to show boundedness in a system with mass-action kinetics
(see first example Section \ref{subsec:example-systems}). However,
for many biological systems it might be challenging to find a suitable
LLF. One solution to this challenge is to leverage computational work
that has been done to find Lyapunov functions \cite{Hafstein2015}.
As shown in Corollary~\ref{cor:LF-is-LLF}, a global Lyapunov function
for the reactions might satisfy the requirements for a LLF. In cases
where a global Lyapunov function does not lead to a suitable LLF,
it might be possible to prove an LLF does not exist (see third example
in Section~\ref{subsec:example-systems}). This suggests the system
has the potential to become unbounded.

A natural future question regarding this work is whether a system
that has a LLF remains bounded in the continuum limit. Here, the bound
obtained depends on the number of spatial compartments and, therefore,
the question of what occurs in the continuum limit is not yet answered.
Additionally, the conditions for boundedness of one type of system
(i.e., the discretized or continuous system) do not satisfy the conditions
to guarantee boundedness of the other system (see \cite{Morgan1990}
for conditions for the continuous system). Research looking at the
continuous system with Dirichlet boundary conditions has found examples
of systems that are bounded with respect to the $L^{1}$ norm but
blow up with respect to the $L^{\infty}$ norm \cite{Pierre2010}.
However, to the best of the author's knowledge, no examples of this
type have been found for systems with positive solutions and homogeneous
Neumann boundary conditions. This suggests that boundedness for the
class of RD systems discussed in the paper might be preserved in the
continuum limit. 

One motivation of this study was to find a class of RD systems that
could be studied using systems biology approaches. When considering
the discretized system it becomes feasible to apply existing systems
biology tools, such as \emph{stoichiometric network analysis} \cite{Clarke1988,Palsson2006,Gianchandani2010}
and \emph{chemical reaction network theory} \cite{Feinberg1979,Craciun2005,Craciun2006a},
to study spatially heterogeneous systems or systems with bounded diffusion-driven
instabilities. We have shown that the results presented here can be
applied to this type of system (see first example in Section \ref{subsec:example-systems}).
In the future we hope to use systems biology tools to study how spatial
features influence system properties (e.g., how does altering the
diffusion ratio between two species affect the the space of possible
reactive fluxes under steady-state conditions). Ultimately, the results
presented here will help us study diffusion-driven instabilities in
complex biochemical systems with variable diffusion and reaction rates.

\section*{Acknowledgments}

The authors would like to thank Prof. Nancy Rodriguez for insightful
discussions regarding this work. 

\bibliographystyle{siamplain}
\bibliography{Papers-BoundedRDSystem}

\appendix

\section{Proofs for the secondary properties of the LLF\label{app-sec:proofOfC1andC2}}

In this section we will prove Corollary \ref{cor:secondary-properties-1}
and \ref{cor:secondary-properties-2}, which guarantee a LLF, $W:\mathbb{R}_{\ge0}^{2}\rightarrow\mathbb{R}_{\ge0}$,
has the additional properties given by \ref{secondary-prop:Mu-L-Mv-L-increasing}--\ref{secondary-prop:Mv-infinity-constant}. 
\begin{proof}
	[Proof of Corollary \ref{cor:secondary-properties-1} \ref{secondary-prop:Mu-L-Mv-L-increasing}]
	Pick $L_{1}$ and $L_{2}$ such that $L_{2}>L_{1}>0$. Using \ref{prop:convex},
	we have that
	\begin{align*}
	M_{u}^{(L_{1})} & =\max_{\|(u,v)\|=L_{1}}\partial_{u}W(u,v)\\
	& =\max_{v\in[0,L_{1}]}\partial_{u}W(L_{1}-v,v)\\
	& <\max_{u\in[0,L_{1}]}\partial_{u}W(L_{2}-v,v)\\
	& \le\max_{u\in[0,L_{2}]}\partial_{u}W(L_{2}-v,v)=M_{u}^{(L_{2})}.
	\end{align*}
	Here, the first inequality holds because $L_{2}>L_{1}$ and $\partial_{uu}W>0$.
	The second inequality holds because we are taking the maximum over
	a larger region. This result proves that $M_{u}^{(L)}$ is monotonically
	increasing. The result for $M_{v}^{(L)}$ follows analogously.
\end{proof}

\begin{proof}
	[Proof of Corollary \ref{cor:secondary-properties-1} \ref{secondary-prop:u-v-underbar-exist}]
	By \ref{prop:radially-unbounded} we know $W(u,0)\rightarrow\infty$
	as $u\rightarrow\infty$. It follows that there exists a constant
	$\text{\ensuremath{\underline{u}}}$ such that $W(\underline{u},0)>W(0,0).$
	Using the Mean Value Theorem and \ref{prop:convex}, we have that
	$\partial_{u}W(\underline{u},0)>0.$ Therefore, by \ref{prop:convex},
	$\partial_{u}W(u,v)>0$ for all $v\ge0$ and $u\ge\underline{u}$.
	The same logic can be used to prove there exists a constant $\underline{v}$
	such that $\partial_{v}W(u,v)>0$ for all $v\ge\text{\ensuremath{\underline{v}}}$
	and $u\ge0$.
\end{proof}

\begin{proof}
	[Proof of Corollary \ref{cor:secondary-properties-1} \ref{secondary-prop:K-exists}]
	Define
	\[
	\hat{M}:=\max_{\|(u,v)\|\le\text{\ensuremath{\underline{u}+\underline{v}}}}W(u,v).
	\]
	By \ref{prop:radially-unbounded} there exist constants $\widetilde{u},\widetilde{v}\in\mathbb{R}_{\ge0}$
	such that
	\begin{equation}
	\begin{aligned}\min_{v\in[0,\underline{v}]}W(u,v) & >\hat{M} &  & \text{for }u\ge\widetilde{u}\\
	\min_{u\in[0,\underline{u}]}W(u,v) & >\hat{M} &  & \text{for }v\ge\widetilde{v}.
	\end{aligned}
	\label{app-eq:mhat-bound}
	\end{equation}
	Note that by definition $\widetilde{u}>\underline{u}+\underline{v}$
	and $\widetilde{v}>\underline{u}+\underline{v}$.
	
	Let $K:=\max\{\underline{K},\widetilde{u}+\widetilde{v}$\} and recall
	that that $M^{(K)}=\max_{\|(u,v)\|=K}W(u,v)$. If $\|(u,v)\|=K$,
	then either $u\ge\widetilde{u}$ or $v\ge\widetilde{v}.$ We will
	show that $M^{(K)}>\hat{M}$ for $u\ge\widetilde{u}$, and the result
	for $v\ge\widetilde{v}$ follows analogously. If $u\ge\widetilde{u}$
	and $v\le\underline{v}$ then using (\ref{app-eq:mhat-bound}) we
	have that $W(u,v)>\hat{M}$. If instead, $v>\underline{v}$, then
	$\partial_{v}W$ is positive and $W(u,v)>W(u,\underline{v})>\hat{M}$.
	Thus, if $\|(u,v)\|=K$ then $W(u,v)>\hat{M}$, and therefore $M^{(K)}>\hat{M}$. 
	
	Next pick $L<K$, and we will prove the claim in the corollary that
	$M^{(L)}<M^{(K)}.$ If $L\le\underline{u}+\underline{v}$ we immediately
	have that $M^{(L)}\le\hat{M}<M^{(K)}$. Alternatively, if $L>\underline{u}+\underline{v}$,
	we have that
	\begin{align*}
	M^{(L)} & =\max_{\|(u,v)\|=L}W(u,v)=\max\left\{ \max_{u\in[0,\underline{u}]}W(u,L-u),\max_{v\in[0,L-\underline{u}]}W(L-v,v)\right\} \\
	& <\max\left\{ \max_{u\in[0,\underline{u}]}W(u,K-u),\max_{v\in[0,L-\underline{u}]}W(K-v,v)\right\} \\
	& \le\max\left\{ \max_{u\in[0,\underline{u}]}W(u,K-u),\max_{v\in[0,K-\underline{u}]}W(K-v,v)\right\} \\
	& =\max_{\|(u,v)\|=K}W(u,v)\\
	& =M^{(K)}.
	\end{align*}
	In this calculation we are breaking apart the the line $\|(u,v)\|=L$
	into two regions. In the region where $u\in[0,\underline{u}]$ we
	are guaranteed that $v>\underline{v}$ and thus, $\partial_{v}W>0$.
	In the other region where $v\in[0,L-\underline{u}]$ we are guaranteed
	that $u>\underline{u}$ and thus $\partial_{u}W>0.$ This leads to
	the first inequality. The second inequality follows because we are
	taking the maximum over a larger region. Note that by definition $K\ge\text{\ensuremath{\underbar{\ensuremath{K}}}}$,
	$\underbar{\ensuremath{u}}$, $\underbar{\ensuremath{v}}$.
\end{proof}

\noindent 
\begin{proof}
	[Proof of Corollary \ref{cor:secondary-properties-2}]We will show
	the proof for \ref{secondary-prop:Mu-infinity-constant}. The proof
	for \ref{secondary-prop:Mv-infinity-constant} follows analogously.
	Recall that
	\begin{equation}
	M_{u}^{(\infty)}(v)=\lim_{u\rightarrow\infty}\partial_{u}W(u,v).\label{app-eq:Mu-infinity-v}
	\end{equation}
	By \ref{prop:convex}, $\partial_{uu}W>0$ and, therefore, $\partial_{u}W(u,v)$
	is monotonically increasing with respect to $u$. This means that
	for a given $v$ the limit given by (\ref{app-eq:Mu-infinity-v})
	either converges and exists or diverges to infinity. Furthermore,
	since $\partial_{uv}W(u,v)\ge0$ we know that $M_{u}^{(\infty)}(v)$
	must be monotonically non-decreasing with respect to $v$. Note that
	if for any $v\in[0,\infty)$, $M_{u}^{(\infty)}(v)=\infty$, then
	by \ref{prop:finite-or-infinite-limits}, $M_{u}^{(\infty)}(v)=\infty$
	for all $v$, and the conclusions of the corollary follow. Therefore,
	in the remainder of the proof we will assume that $M_{u}^{(\infty)}(v)$
	is finite for all $v\in[0,\infty).$ By \ref{secondary-prop:u-v-underbar-exist},
	we then have that $M_{u}^{(\infty)}(v)>0$.
	
	First, we will show that $h(v):=\lim_{u\rightarrow\infty}\partial_{uv}W(u,v)=0$.
	By \ref{prop:convex} and \ref{prop:finite-or-infinite-limits} we
	know $h(v)$ exists and is non-negative. This implies that there exists
	a constant $U>0$ such that if $u>U$ then $\partial_{uv}W(u,v)>h(v)/2$.
	We then have that
	
	\begin{align}
	\int_{U}^{u}\partial_{uv}W(\widetilde{u},v)d\widetilde{u} & \ge\int_{U}^{u}\frac{h(v)}{2}d\widetilde{u}\nonumber \\
	\implies\partial_{v}W & \ge\frac{h(v)}{2}u+C\label{app-eq:hv-bound}
	\end{align}
	where $C$ is a constant. Note that $\limsup_{u\rightarrow\infty}\partial_{v}W$
	must be bounded for any $v\in[0,\infty)$ since 
	\[
	\sup_{u\ge\text{\ensuremath{\underbar{\ensuremath{u}}}}}\left|\frac{\partial_{v}W(u,v)}{\partial_{u}W(u,v)}\right|\ge\limsup_{u\rightarrow\infty}\left|\frac{\partial_{v}W(u,v)}{\partial_{u}W(u,v)}\right|=\frac{\limsup_{u\rightarrow\infty}\left|\partial_{v}W(u,v)\right|}{M_{u}^{(\infty)}(v)}.
	\]
	and by \ref{prop:level-set-tangent-lines} the supremum is finite.
	Taking the limsup of both sides of (\ref{app-eq:hv-bound}) as $u\rightarrow\infty$
	shows that this bound would not hold if $h(v)>0$. Therefore, $h(v)=0$
	for all $v\in[0,\infty).$
	
	Our next goal is to show that $h'(v)=\lim_{u\rightarrow\infty}\partial_{uv}W(u,v)$
	for all $v\in[0,\infty)$. Notice that this relation can be rewritten
	as
	
	\[
	\lim_{h\rightarrow0}\lim_{u\rightarrow\infty}F(u,v,h)=\lim_{u\rightarrow\infty}\lim_{h\rightarrow0}F(u,v,h)
	\]
	where
	
	\[
	F(u,v,h)=\frac{\partial_{u}W(u,v+h)-\partial_{u}W(u,v)}{h}.
	\]
	Thus, we need to show that the limits are interchangeable. 
	
	Since $\partial_{u}W(u,v)$ converges pointwise to $g(v)$ as $u\rightarrow\infty$
	and $\partial_{u}W(u,v)$ is monotonically increasing with respect
	to $u$, by Dini's Monotone Convergence Theorem $\partial_{u}W(u,v)$
	converges uniformly to $g(v)$ for $v\in[0,L]$ where $L$ is an arbitrary
	constant. Therefore $\lim_{u\rightarrow\infty}F(u,v,h)$ exists and
	converges uniformly. We furthermore know that the $\lim_{h\rightarrow0}F(u,v,h)$
	exists and converges pointwise. Thus, by the Moore-Osgood Theorem,
	the limits are interchangeable and the resulting values are equal. 
	
	In conclusion, we have that
	
	\begin{align*}
	\frac{d}{dv}M_{u}^{(\infty)}(v) & =\frac{d}{dv}\lim_{u\rightarrow\infty}\partial_{u}W(u,v)=\lim_{u\rightarrow\infty}\partial_{uv}W(u,v)=0
	\end{align*}
	Thus, for all $v\in[0,L]$, $M_{u}^{(\infty)}(v)$ is constant. Let
	$M_{u}^{(\infty)}:=M_{u}^{(\infty)}(v)$ and note that the upper bound
	$L$ was arbitrary and therefore we have that this equality holds
	for all $v\in[0,\infty)$.General rules for determining whether an
	LLF exists: Proofs\label{sec:General-rules-proofs}
\end{proof}

Below are the proofs for Corollary \ref{cor:LF-is-LLF} and \ref{cor:no-LLF-exists}
in the paper.
\begin{proof}
	[Proof of Corollary \ref{cor:LF-is-LLF}] We can immediately show
	that $W$ satisfies \ref{prop:level-set-tangent-lines} since 
	
	\begin{align*}
	\sup_{v\ge\underline{v}}\left|\frac{\partial_{u}W(u,v)}{\partial_{v}W(u,v)}\right| & \le\frac{w_{1}'(u)}{w_{2}'(v)}\le\frac{w_{1}'(u)}{w_{2}'(\underline{v})}<\infty\\
	\sup_{u\ge\underline{u}}\left|\frac{\partial_{v}W(u,v)}{\partial_{u}W(u,v)}\right| & \le\frac{w_{2}'(v)}{w_{1}'(u)}\le\frac{w_{2}'(v)}{w_{1}'(\underline{u})}<\infty.
	\end{align*}
	Additionally \ref{prop:finite-or-infinite-limits} is satisfied since
	\begin{align*}
	\lim_{u\rightarrow\infty}\partial_{u}W(u,v) & =\lim_{u\rightarrow\infty}w'_{1}(u)=C_{1}\\
	\lim_{v\rightarrow\infty}\partial_{v}H(u,v) & =\lim_{v\rightarrow\infty}w_{2}'(v)=C_{2}\\
	\partial_{uv}W(u,v) & =0
	\end{align*}
	where, due to \ref{prop:convex}, $C_{1}$ and $C_{2}$ are constants
	or infinite.
\end{proof}

\begin{proof}
	[Proof of Corollary \ref{cor:no-LLF-exists}] We will show that
	if there exists a $v$ such that (\ref{eq:f-positive}) is satisfied,
	then no LLF exists. The result for any $u$ and (\ref{eq:g-positive})
	follows analogously. Suppose there exists a LLF for the system. By
	\ref{prop:decreases-with-time}, there exists a $\underline{K}$ such
	that, if $\|(u,v)\|>\underline{K}$, then
	\begin{align}
	\left(\nabla W\right)(f,g)^{T} & =\partial_{u}W(u,v)f(u,v)+\partial_{v}W(u,v)g(u,v)\le0\label{eq:prop1-holds}
	\end{align}
	Pick $v>0$ and consider what happens in the limit as $u\rightarrow\infty$.
	By \ref{prop:decreases-with-time}, \ref{secondary-prop:u-v-underbar-exist},
	and (\ref{eq:f-positive}), there exists $\widetilde{u}$ such that
	if $u>\widetilde{u}$ then $u>\underline{K}$, $\partial_{u}W>0$,
	and $f(u,v)>0$. We therefore have that for $u\ge\widetilde{u}$,
	\[
	\frac{\partial_{v}W(u,v)}{\partial_{u}W(u,v)}\begin{cases}
	\le-\frac{f(u,v)}{g(u,v)}<0 & \text{if }g(u,v)>0\\
	\ge-\frac{f(u,v)}{g(u,v)}>0 & \text{if }g(u,v)<0.
	\end{cases}
	\]
	Note that since $f>0$ and $\partial_{u}W>0$, in order for (\ref{eq:prop1-holds})
	to hold, $g\ne0$. This set of inequalities implies that
	\[
	\left|\frac{\partial_{v}W(u,v)}{\partial_{u}W(u,v)}\right|\ge\left|\frac{f(u,v)}{g(u,v)}\right|.
	\]
	Note that, in the limit as $u\rightarrow\infty$, $\left|f(u,v)/g(u,v)\right|=\infty$
	and, therefore
	\[
	\sup_{u\ge\text{\ensuremath{\underbar{\ensuremath{u}}}}}\left|\frac{\partial_{v}W(u,v)}{\partial_{u}W(u,v)}\right|=\infty.
	\]
	This final equation gives us a contradiction to \ref{prop:level-set-tangent-lines}.
	Therefore, no LLF exists for the reactions.
\end{proof}

\section{Example Systems}

Below we provide details on the simulations and LLF results discussed in Section
\ref{subsec:example-systems}.

\subsection{Parameters for performing simulations\label{sec:simulation-info}}

Here we give the parameters used to perform the simulations shown
in Figure \ref{fig:example-phase-planes}. For simulations with diffusion
the system given by (\ref{eq:sysDisc}) was used where $n=2$. In
the simulation shown in the left panel, $f$ and $g$ are given by
(\ref{eq:reactions-bounded}) where $a=0.1$, $b=1$, $\gamma=150$,
$d=30$, $(u_{1,0},v_{1,0})=(0.8,0.1)$ and $(u_{2,0},v_{2,0})=(2.0,0.7)$.
For the simulation shown in the middle panel, $f$ and $g$ are given
by (\ref{eq:reactions-mutual}) where $a_{1}=-1$, $a_{2}=1$, $b_{1}=1$,
$b_{2}=2$, $c_{1}=1$, $c_{2}=1$, $\gamma=1$, $d=1$, $(u_{1,0},v_{1,0})=(3,0.001)$
and $(u_{2,0},v_{2,0})=(0.001,3)$. For the simulation shown in the
right panel, $f$ and $g$ are given by (\ref{eq:reactions-unbounded})
where $d=1$, $\gamma=1$, $\delta=10$, $u_{1,0}=v_{2,0}=2$, and
$u_{2,0}=v_{1,0}=4$.

\subsection{Existence of Lyapunov-like function for first example\label{app-sec:exampleSystemBounded}}

In this section we prove that the LLF given by (\ref{eq:example-LLF})
has the properties \ref{prop:decreases-with-time}--\ref{prop:finite-or-infinite-limits}
given in Section~\ref{subsec:lyapunov-like-function}. We will go
through each property individually:

\bgroup 
\renewcommand\theenumi{(P\arabic{enumi})} 
\renewcommand\labelenumi{\theenumi}
\begin{enumerate}
	\item We will show that for large enough $\|(u,v)\|$, $\left(\nabla W\right)(f,g)^{T}$
	is negative. We have that 
	\[
	\left(\nabla W\right)(f,g)^{T}=\left(1-\frac{c}{(u+1)^{2}}\right)(a-u+u^{2}v)+\left(2-\frac{1}{(1+v)^{2}}\right)(b-u^{2}v).
	\]
	Combining the terms and using $N(u,v)$ to represent the numerator
	we can write this equation as follows: 
	\[
	\left(\nabla W\right)(f,g)^{T}=\frac{N(u,v)}{(1+u)^{2}(1+v)^{2}}
	\]
	We will show that there exists a value $\text{\ensuremath{\underline{K}}}$
	such that if $\|(u,v)\|>\underline{K}$, then $N(u,v)<0$. To do this
	we will find threshold values for $u$ and $v$ separately. 
	\begin{enumerate}
		\item Let's first consider $u$. After some algebraic manipulates, we rewrite
		$N(u,v)$ as 
		\[
		\begin{aligned}N(u,v) & =(1+u^{2})(a+b-u)+u(2a+2b+c-2u)\\
		& +(2v+2u^{2}v+v^{2}+u^{2}v^{2})(a+2b-u)\\
		& +(2uv+uv^{2})(2a+4b+c-2u)\\
		& -ac-2acv-cu^{2}v-acv^{2}-(2c+2)u^{2}v^{2}\\
		& -4u^{3}v^{2}-2u^{4}v^{2}-u^{2}v^{3}-cu^{2}v^{3}-2u^{3}v^{3}-u^{4}v^{3}.
		\end{aligned}
		\]
		This equation is negative if $u>\widetilde{u}:=a+2b+\frac{c}{2}$. 
		\item Let's next consider $v$ and write $N(u,v)$ as follows 
		\[
		\begin{aligned}N(u,v) & =v^{2}(a+2b-\frac{ac}{2})+u^{2}v^{3}(a+2b-\frac{c}{2})+(a+b-2acv)\\
		& +u(2a+2b+c-2v)+u^{2}(a+b-cv)+v(3a+6b+\frac{c}{2}-\frac{ac}{2}v)\\
		& +uv(4a+8b+2c-v)+u^{2}v^{2}(a+2b-v)+2u^{2}v(a+2b-cv)\\
		& -ac-u-2u^{2}-u^{3}-4u^{2}v-2u^{3}v-(a+2b+\frac{c}{2})v(uv-1)^{2}\\
		& -4u^{2}v^{2}-5u^{3}v^{2}-2u^{4}v^{2}-2u^{3}v^{3}-u^{4}v^{3}.
		\end{aligned}
		\]
		This equation is negative if $c>\max\{(2a+4b)/a,2a+4b\}$ and 
		\[
		\begin{aligned}v & >\widetilde{v}:=\max\left\{ 4a+8b+2c,\frac{6a+12b+c}{ac},\frac{a+2b}{c}\right\} \end{aligned}
		\]
	\end{enumerate}
	Let $\underline{K}:=\widetilde{u}+\widetilde{v}$. If $\|(u,v)\|>\underline{K}$
	then either $u>\widetilde{u}$ or $v>\widetilde{v}$ and, therefore,
	$\left(\nabla W\right)^{T}(f,g)<0$.
	\item Taking the second derivatives of $W$ gives us: 
	\[
	\begin{aligned}\partial_{uu}W & =\frac{2c}{(u+1)^{3}}>0\\
	\partial_{vv}W & =\frac{2}{(v+1)^{3}}>0\\
	\partial_{uv}W & =0.
	\end{aligned}
	\]
	Thus, the desired inequalities are satisfied.
	\item We have that, 
	\[
	W(u,v)=\|(u,v)\|+v+\frac{c}{u+1}+\frac{1}{v+1}\ge\|(u,v)\|.
	\]
	Thus, as $\|(u,v)\|\rightarrow\infty$, $W(u,v)\rightarrow\infty$.
	\item Note, that for this system $\partial_{u}W>0$ if $u\ge c$ and $\partial_{v}W>0$
	if $v\ge0.$ Therefore, we set $\underline{u}=c$ and $\underline{v}=0$.
	For an arbitrary $u>0$, we have that 
	\begin{align*}
	& \sup_{v\ge\underline{v}}\left|\frac{\partial_{u}W(u,v)}{\partial_{v}W(u,v)}\right|=\sup_{v\ge0}\frac{\left|1-\frac{c}{(u+1)^{2}}\right|}{2-\frac{1}{(v+1)^{2}}}\le\left|1-\frac{c}{(u+1)^{2}}\right|<\infty
	\end{align*}
	and for an arbitrary $v>0$, we have that
	\[
	\sup_{u\ge\underline{u}}\left|\frac{\partial_{v}W(u,v)}{\partial_{u}W(u,v)}\right|=\sup_{u\ge c}\frac{2-\frac{1}{(v+1)^{2}}}{1-\frac{c}{(u+1)^{2}}}\le\frac{2-\frac{1}{(v+1)^{2}}}{1-\frac{c}{(c+1)^{2}}}<\infty.
	\]
	Thus, the specified supremums are finite.
	\item Taking the limits specified in the property gives us
	\begin{align*}
	\lim_{u\rightarrow\infty}\partial_{u}W(u,v) & =\lim_{u\rightarrow\infty}\left(1-\frac{c}{(u+1)^{2}}\right)=1\\
	\lim_{v\rightarrow\infty}\partial_{v}W(u,v) & =\lim_{v\rightarrow\infty}\left(2-\frac{1}{(v+1)^{2}}\right)=2\\
	\lim_{u\rightarrow\infty}\partial_{uv}W(u,v) & =\lim_{v\rightarrow\infty}\partial_{uv}W(u,v)=0.
	\end{align*}
	Therefore, all the limits exist and are finite.
\end{enumerate}
\egroup{} 

\subsection{No Lyapunov-like function exists for unbounded examples \label{sm-subsec:noLLFforExample}}

In this section we show that no LLF exists for the two unbounded example
systems. We will first consider the system given by (\ref{eq:reactions-mutual})
and suppose an LLF does exist. By \ref{prop:decreases-with-time}
we have that there exists a $\underline{K}>0$ such that if $\|(u,v)\|\ge\underline{K}$
then
\[
\partial_{u}W(u,v)(u(a_{1}-b_{1}u+c_{1}v))+\partial_{v}W(v(a_{2}+b_{2}u-c_{2}v))\le0.
\]

Suppose $v=\frac{b_{2}}{c_{2}}u$ and $v\ge\underline{v}$, $u\ge\underline{u}$.
We then have that $\partial_{u}W(u,v),\partial_{v}W(u,v)>0$ and
\begin{align*}
\partial_{u}W(u,v)\left(a_{1}u+\left(c_{1}\frac{b_{2}}{c_{2}}-b_{1}\right)u^{2}\right)+\partial_{v}W\left(a_{2}\frac{b_{2}}{c_{2}}u\right) & \le0.\\
\implies a_{1}u+\left(c_{1}\frac{b_{2}}{c_{2}}-b_{1}\right)u^{2} & \le0
\end{align*}
Note that as $u\rightarrow\infty,$the quadratic term dominates and
therefore we require that
\[
c_{1}b_{2}-b_{1}c_{2}\le0.
\]
However, recall that for the system to be strongly mutualistic we
require that $b_{2}c_{1}>b_{1}c_{2}$ and therefore we have a contradiction.
Therefore, no LLF exists for this system.

Next, we will use Corollary~\ref{cor:no-LLF-exists} to show that
no LLF exists for the reactions given by (\ref{eq:reactions-unbounded}).
Suppose we fix a value of $v>0$ and consider what happens in the
limit as $u\rightarrow\infty$. We have that at some point $u>v+1$
and $u>\delta/v-1.$ This leads to the follow inequalities
\begin{align*}
f(u,v) & =uv(u-v)(u+1)-\delta u>0\\
g(u,v) & =uv(v-u)(v+1)-\delta v<0
\end{align*}
and we have that
\[
\lim_{u\rightarrow\infty}\left|\frac{f(u,v)}{g(u,v)}\right|=\lim_{u\rightarrow\infty}\frac{uv(u-v)(u+1)-\delta u}{uv(u-v)(v+1)+\delta v}=\lim_{u\rightarrow\infty}\frac{(u+1)-\frac{\delta}{v(u-v)}}{(v+1)+\frac{\delta}{u(u-v)}}=\infty.
\]
Thus, by Corollary~\ref{cor:no-LLF-exists} no LLF for the system
exists.

\subsection{Unboundedness of example system \label{app-subsec:exampleUnboundedProof}}

In this section, we will show that the discretized RD system given
by (\ref{eq:sysDisc}) with parameters $\gamma=1$, $d=1$, and $n=2$
and reactions given by (\ref{eq:reactions-unbounded}) has the capacity
to become unbounded. We will use symmetric initial conditions (i.e.
$u_{1,0}=v_{2,0}$ and $u_{2,0}=v_{1,0}$). It follows that, due to
the symmetry of the reactions, $u_{1}(t)=v_{2}(t)$ and $u_{2}(t)=v_{1}(t)$
for all $t>0$. Therefore, the system reduces to
\begin{align}
\frac{du}{dt} & =uv(u-v)(u+1)-\delta u+4(v-u)\nonumber \\
\frac{dv}{dt} & =uv(v-u)(v+1)-\delta v+4(u-v)\label{app-eq:symmetric-ODE-system}\\
u(0) & =u_{0}\nonumber \\
v(0) & =v_{0}\nonumber 
\end{align}
where $u_{0}=u_{1,0}=v_{2,0}$ and $v_{0}=u_{2,0}=v_{1,0}.$ We can
then calculate the concentration of species in each compartment as
$u_{1}(t)=v_{2}(t)=u(t)$ and $v_{1}(t)=u_{2}(t)=v(t)$. 

Let's consider (\ref{app-eq:symmetric-ODE-system}) and calculate
how the difference between $u$ and $v$ evolves with time:
\begin{equation}
\begin{aligned}(u-v)_{t} & =uv(u-v)(u+v+2)+(8+\delta)(v-u)\\
& =(u-v)h(u,v).
\end{aligned}
\label{app-eq:dot-u-minus-v}
\end{equation}
where
\[
h(u,v):=uv(u+v+2)-(8+\delta),
\]
Therefore if,
\begin{equation}
\begin{aligned}u-v & >0\\
h(u,v) & >0
\end{aligned}
\label{app-eq:unbounded-IC-requirement}
\end{equation}
then $(u-v)_{t}>0$. 

We will show that there exists a $\delta>0$ such that for arbitrarily
small $\epsilon>0$, $dh/dt$ is positive along the curve
\begin{equation}
h(u,v)=\epsilon.\label{app-eq:curve}
\end{equation}
Therefore, if the initial data satisfies (\ref{app-eq:unbounded-IC-requirement})
then these inequalities will be satisfied for all time.

We first solve (\ref{app-eq:curve}) explicitly for $v$ to obtain
the positive solution
\begin{equation}
v=-1-\frac{1}{2}u+\sqrt{\frac{8+\delta+\epsilon}{u}+1+u+\frac{u^{2}}{4}}.\label{app-eq:curve2}
\end{equation}
Note that this function is symmetric about the line $u=v$ and it
is concave upwards (i.e., $d^{2}v/du^{2}>0$). We will show that there
exists a constant $C$ such that, along the curve given by (\ref{app-eq:curve}),
$u+v\ge C>0$. Suppose it is not the case (i.e., $u+v<C$) and add
$u$ to both sides of (\ref{app-eq:curve2}) to obtain
\[
C>u+v=-1+\frac{1}{2}u+\sqrt{\frac{8+\delta+\epsilon}{u}+1+u+\frac{u^{2}}{4}}.
\]
Using algebraic manipulations, we obtain the following inequality:
\begin{align*}
0 & >(2+C)u^{2}-C(2+C)u+8+\delta+\epsilon.
\end{align*}
The maximum value of $C$ which guarantees that no real solutions
to this equation exist is given by the solution to $0=C^{3}+2C^{2}-32$.
We will define $C$ as the one real root to this equation. We are
then guaranteed that along the curve (\ref{app-eq:curve}), $u+v\ge C$.

Next, we will assume $u-v>0$ and break up the curve given by (\ref{app-eq:curve})
into two regions: $0<u-v\le1$ and $1<u-v$. For both these regions
we will calculate $dh/dt$ along (\ref{app-eq:curve}) and show that,
under certain conditions, it is positive. We will further suppose
that $\epsilon+\delta<1$. 

Suppose $0<u-v\le1$. In order to show that $dh/dt$ is increasing
along the curve, we will find a lower bound for the value of $u+v$
and translate that into an upper bound on the value of $uv$. Due
to the positive second derivative and symmetry of (\ref{app-eq:curve2}),
we know the maximum value of $u+v$ along the curve occurs when $v=u-1$.
We will call this point $(u_{max},v_{max}).$ Note that $C\le u_{max}+v_{max}=2u_{max}-1$,
and therefore $u_{max}\ge(C+1)/2$. We will next calculate an upper
bound on $u_{max}$. We have that
\[
\begin{aligned}h(u_{max},v_{max}) & =\epsilon\\
u_{max}v_{max}(u_{max}+v_{max}+2) & =\epsilon+\delta+8\\
u_{max}(u_{max}-1)(2u_{max}+1) & =\epsilon+\delta+8\\
u_{max}-1 & =\frac{\epsilon+\delta+8}{u_{max}(2u_{max}+1)}\\
u_{max}-1 & <\frac{9}{\frac{(C+1)}{2}(C+2)}=\frac{18}{(C+1)(C+2)}\\
u_{max} & <\frac{18}{(C+1)(C+2)}+1
\end{aligned}
\]
It then immediately follows that $v_{max}<\frac{18}{(C+1)(C+2)}.$
We then have that, along the curve given by (\ref{app-eq:curve})
\[
uv=\frac{\epsilon+\delta+8}{u+v+2}>\frac{\epsilon+\delta+8}{u_{max}+v_{max}+2}>\frac{8}{\frac{36}{(C+1)(C+2)}+3}=:A.
\]
Finally, we consider and bound the derivative of $h$ along the curve.
Recall that $u+v>C$ and $\delta+\epsilon<1$. We have that along
\ref{app-eq:curve} when $0<u-v\le1$
\begin{align*}
\frac{dh(u,v)}{dt} & =(u-v)^{2}(2+u+v)(4-uv)+2uv+\left(uv\right)^{2}(u-v)^{2}-\delta uv(4+3(u+v))\\
& \ge(u-v)^{2}(4(2+u+v)-(8+\delta+\epsilon))+uv(2-\delta(4+3(u+v)))\\
& \ge(u-v)^{2}(4C-1)+uv(2+2\delta-3\delta(u+v+2))\\
& \ge uv(2+2\delta)-3\delta(8+\epsilon+\delta)\\
& \ge A(2+2\delta)-27\delta\\
& \ge2A-(27-2A)\delta.
\end{align*}
Thus, if $\delta<\frac{2A}{27-2A}$ then the derivative is increasing. 

Suppose $u-v>1.$ The derivative of $h$ along the curve is bounded
as follows:
\begin{align*}
\frac{dh(u,v)}{dt} & =(u-v)^{2}(2+u+v)(4-uv)+2uv+\left(uv\right)^{2}(u-v)^{2}-\delta uv(4+3(u+v))\\
& \ge(4(2+u+v)-(8+\delta+\epsilon))-\delta uv(4+3(u+v)))\\
& \ge(4C-1)-\delta uv(-2+3(u+v+2)))\\
& \ge(4C-1)+2\delta uv-3\delta(8+\epsilon+\delta)\\
& \ge(4C-1)-3\delta(8+\epsilon+\delta)\\
& \ge(4C-1)-27\delta
\end{align*}
So the derivative is positive if $\delta<\frac{4C-1}{27}$. 

Therefore, for small enough $\delta$, if the system satisfies (\ref{app-eq:unbounded-IC-requirement})
it will continue to do so for all time. Numerically, we determined
that if $\delta\le0.13$ then the necessary conditions are satisfied.

Using (\ref{app-eq:dot-u-minus-v}) and assuming the initial data
satisfies (\ref{app-eq:unbounded-IC-requirement}), we have that for
small enough $\epsilon>0$.
\begin{align*}
(u-v)_{t} & \ge\epsilon(u-v).
\end{align*}
Using Grönwall's inequality, we then have that

\[
u(t)-v(t)\ge(u_{0}-v_{0})e^{\epsilon t}.
\]
Therefore, since $v(t)\ge0$, $u(t)\rightarrow\infty$ as $t\rightarrow\infty$.
Note that analogously we could pick initial conditions where $v_{0}>u_{0}$
and $u_{0}v_{0}(u_{0}+v_{0}+2)>8+\delta$ and this would lead to a
blow up of $v(t)$. 
\end{document}